\documentclass[preprint,compress,12pt]{elsarticle}

\usepackage{graphicx}
\usepackage{pifont,latexsym,ifthen,amsthm,rotating,calc,textcase,booktabs}
\usepackage{amsfonts,amssymb,amsbsy,amsmath}
\usepackage{mathrsfs}
\usepackage{appendix}

\newtheorem{theorem}{Theorem}[section]
\newtheorem{lemma}[theorem]{Lemma}

\newtheorem{proposition}[theorem]{Proposition}

\usepackage{natbib}
\newtheorem{rhp}{RH Problem}
\numberwithin{equation}{section}
\usepackage{todonotes}
\usepackage{float}
\usepackage{subfigure}
\usepackage{enumitem}
\usepackage[colorlinks,linkcolor=blue]{hyperref}
\hypersetup{
    colorlinks=true, 
    linktoc=all,     
    linkcolor=blue,  
    citecolor=blue
}

\begin{document}

\begin{frontmatter}
\title{Existence of global solutions to the Fokas-Lenells equation with arbitrary spectral singularities}

\author[inst1]{Yuan Li}
\author[inst2]{Qiaoyuan Cheng}
\author[inst1]{Engui Fan$^{*,}$}

\address[inst1]{ School of Mathematical Sciences and Key Laboratory of Mathematics for Nonlinear Science, Fudan University, Shanghai, 200433, China}
\address[inst2]{School of Mathematics, Shanghai University of Finance and Economics, Shanghai, 200433, China\\
* Corresponding author and e-mail address: faneg@fudan.edu.cn  }

\begin{abstract}
We establish the global existence of solutions to the Fokas-Lenells equation for any initial data in 
a weighted Sobolev space  $H^{3}(\mathbb{R})\cap H^{2,1}(\mathbb{R})$.
This result removes  all spectral restrictions  on  the initial data required in our previous work.
 The proof primarily relies on the inverse scattering transform formulated as new  Riemann-Hilbert problems  and Zhou's $L^{2}$-Sobolev bijectivity theory.
\end{abstract}

\begin{keyword}
   Fokas-Lenells equation  \sep Global solutions \sep Inverse scattering transform \sep Riemann-Hilbert problem  \sep Cauchy projection operator.  

  \textit{Mathematics Subject Classification: }35P25; 35Q15; 35G25; 35A01.

  \end{keyword}
\end{frontmatter}
\tableofcontents

\section{Introduction}
The purpose of this paper is to study the existence of global solutions to the Cauchy problem for the Fokas-Lenells (FL) equation
\begin{align}
&u_{tx}-u-i|u|^{2}u_{x}=0,\quad x\in\mathbb{R},\label{fl}\\
&u(x,t=0)=u_{0}(x),\label{initial}
\end{align}
where the initial data $u_{0}\in H^{3}(\mathbb{R})\cap H^{2,1}(\mathbb{R})$.
  The FL equation (\ref{fl}), as  an integrable generalization of the well-known  nonlinear Schr\"{o}dinger (NLS) equation,   was first  derived by Fokas \cite{Fokas} using the bi-Hamiltonian methods. This relationship is similar to that between the Camassa-Holm equation and the Korteweg-de Vries equation \cite{ch}. The  equation \eqref{fl} is an equivalent and simpler form of the original version, which can be derived through a simple change of variables \cite{Lenells}. When certain higher-order nonlinear effects are taken into account, this FL equation arises as a model for nonlinear pulse propagation in monomode optical fibers \cite{Lenells}.

Lenells and Fokas employed the bi-Hamiltonian structure to derive the first few conservation laws and the Lax pair for the FL equation \cite{FL}. Building on this integrable structure, they solved both the initial-value problem \cite{FL} and the initial-boundary value problem \cite{FLibvp}, and implemented the dressing method \cite{FLdr}. The bright and dark solitons of the FL equation were constructed through the bilinear transformation method by Matsuno \cite{bright,dark}. In the presence of a perturbation, the bright soliton propagation of the FL equation was investigated in \cite{pertu}. For initial data in the Schwartz space and a weighted Sobolev space, the long-time asymptotic behavior of solutions to the FL equation was obtained via the Deift-Zhou method \cite{xuj} and the $\overline{\partial}$-steepest descent method \cite{Chengasy}, respectively.

In recent years, the issue of global well-posedness for completely integrable equations has been a subject of significant research interest \cite{Bah1,Bah2,Ha1,Ha2,Jen1,Jen2,Jenkins,Killip1,Killip2,Liu16,Pel1,Pel2}. Within this context, the inverse scattering technique plays a vital role \cite{Jen1,Jen2,Jenkins,Liu16,Pel1,Pel2}. In 1998, Zhou established the $L^{2}$-Sobolev space bijectivity for the Zakharov-Shabat (ZS) spectral problem \cite{Zhou1998}, which forms the rigorous theoretical foundation for analyzing global well-posedness via the inverse scattering transform method.
For initial data not supporting eigenvalues and spectral singularities,  Pelinovsky and Shimabukuro \cite{Pel2} constructed a global solution in $H^{2}(\mathbb{R})\cap H^{1,1}(\mathbb{R})$ for the derivative NLS equation associated with the Kaup-Newell (KN) spectral problem. They subsequently applied the B\"{a}cklund transformation to extend the results to initial data including eigenvalues \cite{Pel1}.
More recently, Jenkins et al. \cite{Jenkins} established the global well-posedness of the derivative NLS equation for any initial data $u_{0}$
in the stronger space $H^{2}(\mathbb{R})\cap L^{2,2}(\mathbb{R})$. This marks the first breakthrough beyond the long-standing requirement of small initial data norms for the global well-posedness of the derivative NLS equation. Moreover, this achievement clearly highlights the significant advantage of the inverse scattering method in addressing the global well-posedness for the derivative NLS equation.

The  studies on the well-posedness of the FL equation remain scarce.  Prior work was largely confined to that of Fokas and Himonas, who established local well-posedness for its periodic initial-value problem in earlier years  \cite{FLci}. The presence of the mixed space-time derivative term $u_{tx}$ in the FL equation \eqref{fl} makes it extremely difficult to consider the well-posedness problem directly from the equation itself by using energy estimate method. Consequently, a natural idea is to exploit the FL equation's Lax pair structure and employ inverse scattering techniques to investigate its global well-posedness of the FL equation on the line.
However   when applying this method,
the scattering data $a(k)$  associated with given initial data $u_0(x)$   contains  two kinds  of  zeros.  The  zeros of $a(k)$ in $\mathbb{C}\setminus(\mathbb{R}\cup i \mathbb{R})$ and on $\mathbb{R}\cup i \mathbb{R}$ are referred to as    the eigenvalues   and spectral singularities of the spectral problem  \eqref{laxx}, respectively.
Consequently, these zeros  will cause  singularities of the  matrix-valued   function $n(x,k)$    defined  by \eqref{rhba} on the complex plane $\mathbb{C}$.
In recent years, we have   employed the inverse scattering transform  method to establish the global well-posedness of the FL equation in $H^{3}(\mathbb{R})\cap H^{2,1}(\mathbb{R})$ for initial data without eigenvalues and spectral singularities \cite{Chengglo}. Later, through the Darboux transformation, we refined the result to allow eigenvalues while avoiding spectral singularities \cite{Chengsoli}.  In the present work, by combining the inverse scattering transform method \cite{Beals} with Zhou's $L^{2}$-bijectivity theory \cite{Zhou1998}, we remove all spectral restrictions on both
 eigenvalues and spectral singularities, establishing global well-posedness of the FL equation \eqref{fl} for any initial data   $u_0(x)\in H^{3}(\mathbb{R})\cap H^{2,1}(\mathbb{R})$.
  To overcome the  difficulty that both  eigenvalues    and  spectral singularities   bring,  we use the distribution properties of the zeros  of  scattering data $a(k)$ to construct a new piecewise analytic function $N(x,k)$ defined by  \eqref{ahkf}  that effectively removes all eigenvalues  and spectral singularities.

\subsection{Main results}
The main results are stated in the following theorem:
\begin{theorem}\label{thm}
For any initial data $u_{0}\in H^{3}(\mathbb{R})\cap H^{2,1}(\mathbb{R})$, there exists a solution
\begin{equation*}
u\in C\left(\left[-T,T\right];H^{3}(\mathbb{R})\cap H^{2,1}(\mathbb{R})\right)
\end{equation*}
to the Cauchy problem \eqref{fl}--\eqref{initial} for every $T>0$. Moreover, the map
\begin{equation*}
H^{3}(\mathbb{R})\cap H^{2,1}(\mathbb{R})\ni u_{0}\longmapsto u \in C\left([-T,T]; H^{3}(\mathbb{R})\cap H^{2,1}(\mathbb{R})\right)
\end{equation*}
is Lipschitz continuous.
\end{theorem}

Here we remak that although  there exists a deep relationship between the derivative NLS equation and the FL equation, the latter being the first negative flow of the integrable hierarchy of the former \cite{Kundu,Lenells},
 the inverse scattering analysis for the FL equation meets technical  difficulties   compared to that for the derivative NLS equation, where  only one reconstruction formula was  required \cite{Jenkins}.
While we have to construct three distinct reconstruction formulas in this paper.
We elaborate on the necessity of the two additional reconstruction formulas
through the following two aspects:

Firstly, for the FL equation, the limit of $N(x,k)$ as $k\to\infty$ reconstructs only $u_{x}$ (see \eqref{rei}), thus the reconstruction of the original  potential $u$ necessarily requires asymptotic analysis at $k\to0$ (see \eqref{rez}). To estimate the reconstructed potential on the $z$-plane, we introduce the first row-vector-valued function $P(x,z)$  defined in \eqref{map} to  derive   two $z$-plane reconstruction formulas \eqref{recp1} and \eqref{recp2}.

Secondly, for initial data $u_{0}\in H^{3}(\mathbb{R})\cap H^{2,1}(\mathbb{R})$, the direct scattering mapping established in \cite{Chengglo} yields that scattering data $G^{(2)}_{\pm}-I$ belongs to $H^{1}(\partial\widetilde{\Omega}^{\pm})\cap L^{2,1}(\partial\widetilde{\Omega}^{\pm})$, see Proposition \ref{timep}. This implies that the scattering data reside in a space with lower regularity compared to the initial data space. However, under this condition, the reconstruction formula \eqref{recp1} can only guarantee that the recovered potential $u_{x}$ belongs to $H^{1}(\mathbb{R})\cap L^{2,1}(\mathbb{R})$, according to Zhou's $L^{2}$-bijectivity theory \cite{Zhou1998}.  Fortunately, we discover that it is possible to overcome this limitation without enhancing the regularity demands on the initial data, as required in \cite{Jenkins}, by constructing a new reconstruction formula. Consequently, we construct the second row-vector-valued function $H(x,z)$ (see \eqref{mah}) and subsequently derive an additional reconstruction formula \eqref{sfs}, thereby obtaining $u_{xx}\in H^{1}(\mathbb{R})\cap L^{2,1}(\mathbb{R})$.
\subsection{Framework of the proof}
This paper is organized as follows.

In Section \ref{sec2}, we focus on the direct scattering transform for the spectral problem \eqref{laxx}. In Section \ref{sub21}, beginning with the KN spectral problem on the $k$-plane, we introduce the fundamental notations and properties used throughout the paper. In Section \ref{sub22}, a novel piecewise analytic matrix-valued function $N(x,k)$ is introduced to remove the eigenvalues and spectral singularities of $n(x,k)$. Section \ref{sub23}  applies a crucial transformation to convert the original KN spectral problem into the ZS spectral problem on the $z$-plane.

Section \ref{sec3} establishes the Riemann-Hilbert (RH) problems and analyzes their solvability.
In Section \ref{sub31}, a classical matrix RH problem \ref{RH1} on the $k$-plane is constructed.
The objective of Section \ref{sub32} is to construct the first type of vector RH problems on the $z$-plane. Using the transformation \eqref{trans}, we formulate the first vector RH problem \ref{RH2} and obtain the first reconstruction formula \eqref{recp1}. It is verified by Proposition \ref{vuyf} that the matching condition at the self-intersection points $\pm R_{\infty}^{2}$ is satisfied by the jump matrix $G(z)$. Subsequently, we construct the RH problems \ref{RHp1} and \ref{RHp2} to achieve the decay of the jump matrix at self-intersections and unify its triangular structure.
To estimate the reconstructed potential, we establish the second type of vector RH problem \ref{RH3} to derive another necessary reconstruction formula \eqref{sfs} in Section \ref{sub33}.  Section \ref{sub34} addresses the time evolution of the jump matrices, and Section \ref{sub35} provides a rigorous proof of the solvability for the RH problems.

Section \ref{sec4} makes the estimate  of the reconstructed potential. By employing the $t$-Lax pair \eqref{psit}, we recover the original potential $u$ at $z\to0$ in Section \ref{sub41}. The relationship between the RH problems and the corresponding Beals-Coifman solutions is used to derive the final reconstruction formulas \eqref{rrp}--\eqref{rru}. Subsequently, Section \ref{sub42} obtains  the estimates for the reconstructed potential, while the proof of Theorem \ref{thm} is presented in Section \ref{sub43}.

\subsection{Notations}

Here we present some notations used through out this paper.
\begin{enumerate}[label=(\arabic*)]
  \item  A weighted $L^{p}$ space is defined by
         \begin{equation*}
         L^{p,s}(\mathbb{R})=\left\{f\in L^{p}(\mathbb{R}): \langle \cdot\rangle^{s}f\in L^{p}(\mathbb{R})\right\}, \quad \langle \cdot\rangle=\sqrt{1+|\cdot|^{2}}.
         \end{equation*}
  \item A Sobolev space is defined by
         \begin{equation*}
         H^{k}(\mathbb{R})=\left\{f\in L^{2}(\mathbb{R}): \langle \cdot\rangle^{k}\widehat{f}\in L^{2}(\mathbb{R})\right\},\quad k\geq0,
         \end{equation*}
         where $\widehat{f}$ is the Fourier transform of $f$.
  \item A weighted Sobolev space is defined by
         \begin{equation*}
         H^{k,s}(\mathbb{R})=\left\{f\in L^{2}(\mathbb{R}): \langle \cdot\rangle^{s}\partial^{j}f\in L^{2}(\mathbb{R}), \, j=0,1,\cdots,k\right\}.
         \end{equation*}
  \item  The $L^{p}$ matrix norm of the matrix function $U$ are denoted as $\|U\|_{L^{p}}:=\left\| \left|U\right| \right\|_{L^{p}}$, where $|U|=\sqrt{\operatorname{tr}(U^{\dag}U)}$.
  \item Let $\Gamma$ be a self-intersecting contour that can be decomposed as $\Gamma=\Gamma_{1}\cup\cdots\cup\Gamma_{n}$, where each $\Gamma_{j}$ is non-self-intersecting and smooth. The Sobolev space $H^{k}(\Gamma)$ is defined by $f|_{\Gamma_{j}}\in H^{k}(\Gamma_{j})$ for all $j$. In our framework, each $\Gamma_{j}$ is parameterizable over $\mathbb{R}$.
  \item If $\Omega$ is a connected open region and $\partial\Omega$ is the nonsmooth boundary of $\Omega$, then $f\in H^{k}(\partial\Omega)$ means that $f$ is $ H^{k}$ on each smooth piece and $f$ matches from two sides to the order $k-1$ at each nonsmooth point. If $\Omega$ is a disconnected region, then $f\in H^{k}(\partial\Omega)$ implies that $f\in H^{k}(\partial\Omega_{i})$ for every connected component $\Omega_{i}$ of $\Omega$.

\end{enumerate}

\section{Direct scattering transform}\label{sec2}
Building upon the framework established in our previous work \cite{Chengglo}, we analyze the direct scattering transform for the FL equation \eqref{fl}.
\subsection{KN spectral problem on the $k$-plane}\label{sub21}
The FL equation \eqref{fl} admits a Lax pair \cite{Chengglo}
\begin{align}
&\phi_{x}=\frac{i}{2}k^{2}\sigma_{3}\phi+k\,U_{x}\phi,\label{laxx}\\
&\phi_{t}=\frac{i}{2}k^{-2}\sigma_{3}\phi+i\sigma_{3}\left(U^{2}-k^{-1}U\right)\phi,\label{laxt}
\end{align}
where
\begin{equation*}
\sigma_{3}=\begin{pmatrix}1&0\\
0&-1\end{pmatrix}, \quad U=\begin{pmatrix}0&u\\
-\bar{u}&0\end{pmatrix}.
\end{equation*}
Let
$$\psi=\phi\, e^{-\frac{i}{2}(k^{2}x+k^{-2}t)\sigma_{3}}, $$
then the Lax pair \eqref{laxx}--\eqref{laxt}  become
\begin{align}
&\psi_{x}=\frac{i}{2}k^{2}\left[\sigma_{3},\psi\right]+k\,U_{x}\psi,\label{psix} \\
&\psi_{t}=\frac{i}{2}k^{-2}\left[\sigma_{3},\psi\right]+i\sigma_{3}\left(U^{2}-k^{-1}U\right)\psi,\label{psit}
\end{align}
where $[\sigma_{3}, \psi]=\sigma_{3}\psi-\psi\sigma_{3}$  and \eqref{psix} is a KN type spectral problem because of  the multiplication by $k$ in the matrix potential $U_{x}$.

For  direct scattering, we first consider the spectral problem \eqref{psix}, therefore as usual we omit the time variable $t$.
Actually, \eqref{psix} can be expressed in the form of an integral equation
\begin{equation*}
\psi(x,k)=e^{\frac{i}{2}k^{2}x\widehat{\sigma}_{3}}\delta(k)+\int_{\beta}^{x}e^{\frac{i}{2}k^{2}(x-y)\widehat{\sigma}_{3}}k\,U_{y}(y)\psi(y,k)dy,
\end{equation*}
where $e^{\frac{i}{2}k^{2}x\widehat{\sigma}_{3}}\delta(k)=e^{\frac{i}{2}k^{2}x\sigma_{3}}\delta(k) e^{-\frac{i}{2}k^{2}x\sigma_{3}}$, $\delta(k)$ is a matrix function independent of $x$,  and the parameter $\beta$ can be chosen differently for different entries of the matrix. Especially, the solution $\psi^{\pm}(x,k)$ to equation \eqref{psix} with the asymptotic behavior
$\psi^{\pm}(x,k)\rightarrow I, x\rightarrow\pm\infty$ admits the following representation
\begin{equation}\label{vo}
\psi^{\pm}(x,k)=I+\int_{\pm\infty}^{x}e^{\frac{i}{2}k^{2}(x-y)\widehat{\sigma}_{3}}k\,U_{y}(y)\psi^{\pm}(y,k)dy.
\end{equation}
Clearly, $\psi^{\pm}(x,0)=I$. Let $\psi^{\pm}=\left[\psi_{1}^{\pm},\psi_{2}^{\pm}\right]$, where the subscripts ``1'' and ``2'' denote the first and second columns of $\psi^{\pm}$,  respectively.

 For every $x\in\mathbb{R}$, it can be shown that  $\psi_{1}^{+}(x,k)$ and $\psi_{2}^{-}(x,k)$ are analytic for $k\in D^{+}$, while $\psi_{1}^{-}(x,k)$ and $\psi_{2}^{+}(x,k)$ are analytic for $k\in D^{-}$, here
\begin{equation}
D^{+}=\left\{k\in\mathbb{C}:\text{Im} k^{2}>0\right\},\quad
D^{-}=\left\{k\in\mathbb{C}:\text{Im} k^{2}<0\right\}.
\end{equation}
Moreover, $\psi^{\pm}$ satisfy the symmetry relations
 \begin{equation}\label{sym}
\psi^{\pm}(k)=\sigma_{2}\overline{\psi^{\pm}(\overline{k})}\sigma_{2},\quad\psi^{\pm}(k)=\sigma_{3}\psi^{\pm}(-k)\sigma_{3}.
\end{equation}

The scattering matrix $S(k)$ associated with the spectral problem \eqref{psix} is given by
\begin{equation}\label{sca}
\psi^{-}(x,k)=\psi^{+}(x,k)e^{\frac{i}{2}k^{2}x\widehat{\sigma}_{3}}S(k),\quad k\in \mathbb{R}\cup i\mathbb{R},
\end{equation}
where
\begin{equation}
S(k)=\begin{pmatrix}a(k)&b(k)\\c(k)&d(k)\end{pmatrix},\quad \det S(k)=1.
\end{equation}
From \eqref{vo} and \eqref{sym}, we have
\begin{equation}\label{vos}
S(k)=I+\int_{\mathbb{R}}e^{-\frac{i}{2}k^{2}y\widehat{\sigma}_{3}}k\,U_{y}(y)\psi^{-}(y,k)dy
\end{equation}
and
\begin{equation}\label{syms}
S(k)=\sigma_{2}\overline{S(\overline{k})}\sigma_{2},\quad S(k)=\sigma_{3}S(-k)\sigma_{3}.
\end{equation}

For scattering data $a(k)$ and $d(k)$, they can be analytically extended to the $D^{-}$ and $D^{+}$ (see Figure \ref{f1}), respectively. The symmetry
relations \eqref{syms} imply that
\begin{equation}
d(k)=\overline{a(\overline{k})},\quad c(k)=-\overline{b(\overline{k})}.
\end{equation}
Moreover, $a(k)$ and $d(k)$ are even functions in $k$, $b(k)$ and $c(k)$ are odd functions in $k$. In particular, $a(0)=d(0)=1$, $b(0)=c(0)=0$.

\begin{figure}
\begin{center}
\begin{tikzpicture}
\filldraw[color=yellow!10](0,0) rectangle (3,3);
\filldraw[color=yellow!10](0,0) rectangle (-3,-3);
\draw [ ](-3,0)--(3,0)  node[right, scale=1] {$\mathbb{R}$};
\draw [ ](0,-3)--(0,3)  node[above, scale=1] {$i\mathbb{R}$};
\draw[fill,black] (0,0) circle [radius=0.03];
\draw [-latex](1.5,0)--(1.7,0);
\draw [-latex](-1.4,0)--(-1.6,0);
\draw [-latex](0, 1.6)--(0,1.5);
\draw [-latex](0, -1.6)--(0,-1.5);
\node at (0.2, -0.2 )  { $0$};
 \node at (1.5, 1.5)  {$D^+$};
 \node at (-1.5, -1.5 )  {$D^+$};
\node at (1.5, -1.5 )  {$D^-$};
 \node at (-1.5, 1.5 )  {$D^-$};
\end{tikzpicture}
\end{center}
\caption{Jump contour $\Gamma_{0}$ and analytic domains $D^{\pm}$ for $n(x,k)$.}
\label{f1}
 \end{figure}
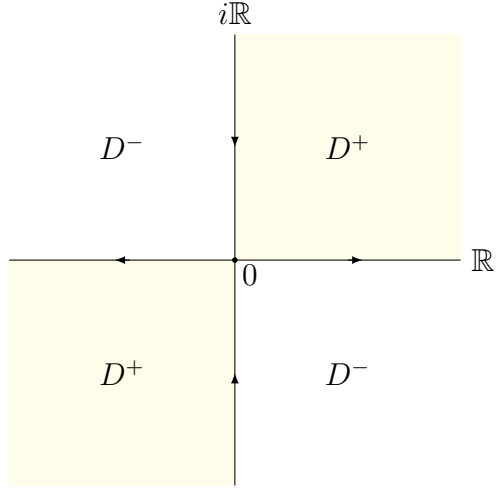
\subsection{Removing eigenvalues and spectral singularities}\label{sub22}
Using the analytic properties of $\psi^{\pm}(x,k)$, $a(k)$ and $d(k)$, we construct a piecewise analytic matrix-valued function
\begin{equation}\label{rhba}
n(x,k)=\left\{
\begin{aligned}
&\left[\psi_{1}^{+}(x,k),\frac{\psi_{2}^{-}(x,k)}{d(k)}\right],\quad k\in D^{+},\\
&\left[\frac{\psi_{1}^{-}(x,k)}{a(k)},\psi_{2}^{+}(x,k)\right],\quad k\in D^{-}.
\end{aligned}
\right.
\end{equation}
From the scattering relation \eqref{sca}, this function satisfies the jump condition
\begin{equation}
n_{+}(x,k)=n_{-}(x,k)e^{\frac{i}{2}k^{2}x\widehat{\sigma}_{3}}\,v(k), \quad k \in \Gamma_{0},
\end{equation}
with
\begin{equation*}
v(k)=\begin{pmatrix}1&-\overline{r(\overline{k})}\\-r(k)&1+r(k)\overline{r(\overline{k})}\end{pmatrix},
\end{equation*}
where the reflection coefficient $r(k)$ is defined by
\begin{equation}
r(k)=\frac{c(k)}{a(k)}.
\end{equation}
Here, the jump contour $\Gamma_{0}=\left\{k\in\mathbb{C}:\text{Im} k^{2}=0\right\}$ and the analytic domains $D^{\pm}$ are characterized in Figure \ref{f1}. However, the scattering data $a(k)$ and $d(k)$ may possess zeros in their respective analytic domains, which results in $n(x,k)$ possessing singularities on the complex plane $\mathbb{C}$.

Proceeding from the properties of the scattering data $a(k)$ obtained in \cite{Chengglo}, we implement two analytical steps: (i) As $k\rightarrow\infty$, the limit of $a(k)$ exists and is non-zero, so we can choose a large enough $R_{\infty}$ such that $a(k)$ has no zeros in $\mathbb{C}\setminus B(0,R_{\infty})$, with $B(0,R_{\infty})=\{k\in\mathbb{C}:|k|<R_{\infty}\}$.
(ii) Under the small norm condition, the scattering data $a(k)$ is never zero, and thus we can select an appropriate $x_{0}\in\mathbb{R}$ such that
\begin{equation*}
\sup_{k\in B(0,R_{\infty})}\left\|kU_{x}\chi_{(x_{0},+\infty)}\right\|_{L^{1}}\ll1.
\end{equation*}
By replacing the potential $u$ in \eqref{psix} with the cut-off potential $u_{x_{0}}=u\chi_{(x_{0},+\infty)}$, we can construct a matrix-valued function $n_{0}(x,k)$ analogous to $n(x,k)$
\begin{equation}
n_{0}(x,k)=\left\{
\begin{aligned}
&\left[\psi_{01}^{+}(x,k),\frac{\psi_{02}^{-}(x,k)}{d_{0}(k)}\right],\quad k\in D^{+},\\
&\left[\frac{\psi_{01}^{-}(x,k)}{a_{0}(k)},\psi_{02}^{+}(x,k)\right],\quad k\in D^{-},
\end{aligned}
\right.
\end{equation}
where the subscript ``$0$'' denote the eigenfunctions and scattering data associated with $u_{x_0}$. Here, $n_{0}(x,k)$ has no singularities on the complex plane.

It follows form \eqref{vo} and \eqref{vos} that
\begin{equation*}
\left[\psi_{1}^{+}(x,k), \frac{\psi_{2}^{-}(x,k)}{d(k)}\right]=I+\int_{\beta}^{x}e^{\frac{i}{2}k^{2}(x-y)\widehat{\sigma}_{3}}k\,U_{y}(y)\left[\psi_{1}^{+}(y,k),\frac{\psi_{2}^{-}(y,k)}{d(k)}\right]dy,
\end{equation*}
where $\beta=-\infty$ for the (1,2) entry and $\beta=+\infty$ for the (1,1), (2,1) and (2,2) entries, and
\begin{equation*}
\left[\frac{\psi_{1}^{-}(x,k)}{a(k)}, \psi_{2}^{+}(x,k)\right]=I+\int_{\beta}^{x}e^{\frac{i}{2}k^{2}(x-y)\widehat{\sigma}_{3}}k\,U_{y}(y)\left[\frac{\psi_{1}^{-}(y,k)}{a(k)}, \psi_{2}^{+}(y,k)\right]dy,
\end{equation*}
where $\beta=-\infty$ for the (2,1) entry and $\beta=+\infty$ for the (1,1), (1,2) and (2,2) entries. Thus, the $n(x,k)$ defined by \eqref{rhba} satisfies \eqref{psix} for $k\in\mathbb{C}\setminus B(0,R_{\infty})$.  Define
\begin{equation}
n_{2}(x,k):=n_{1}(x,k)\,e^{\frac{i}{2}k^{2}(x-x_{0})\widehat{\sigma}_{3}}\,n_{0}(x_{0},k),
\end{equation}
where $n_{1}(x,k)$ is a solution of the integral equation
\begin{equation*}
n_{1}(x,k)=I+\int_{x_{0}}^{x}e^{\frac{i}{2}k^{2}(x-y)\widehat{\sigma}_{3}}k\,U_{y}(y)n_{1}(y,k)dy.
\end{equation*}
Then, we can verify that $n_{2}(x,k)$ also satisfies \eqref{psix}.  From \eqref{vo} and \eqref{vos}, we obtain $n(x,k)\rightarrow I$ as $x\rightarrow +\infty$.
Moreover, $n_{2}(x,k)\rightarrow I$ as $x\rightarrow +\infty$ since $n_{2}(x,k)=n_{0}(x,k)$ for all $x\geq x_{0}$.

\begin{figure}
\begin{center}
\begin{tikzpicture}
 \draw [fill=pink,ultra thick,color=yellow!10] rectangle (3,3);
 \draw [fill=blue,ultra thick,color=yellow!10] (0,0) rectangle (-3,-3);
\filldraw[white](0,0)-- (1.8,0) arc (0:90:1.8);
\filldraw[white](0,0)-- (-1.8,0) arc (180:270:1.8);
\filldraw[color=yellow!10](0,0)-- (0,-1.8) arc (270:360:1.8);
\filldraw[color=yellow!10](0,0)-- (0,1.8) arc (90:180:1.8);
\draw [ ](-3,0)--(3,0)  node[right, scale=1] {$\mathbb{R}$};
\draw [ ](0,-3)--(0,3)  node[above, scale=1] {$i\mathbb{R}$};
\draw[fill,black] (1.8,0) circle [radius=0.03];
\draw[fill,black] (-1.8,0) circle [radius=0.03];
\draw[fill,black] (0,1.8) circle [radius=0.03];
\draw[fill,black] (0,-1.8) circle [radius=0.03];
\draw[fill,black] (0,0) circle [radius=0.03];
\draw [-latex](1,0)--(0.9,0);
\draw [-latex](-1,0)--(-0.9,0);
\draw [-latex](0,0.9)--(0,1);
\draw [-latex](0,-0.9)--(0,-1);
\draw [-latex](2.5,0)--(2.6,0);
\draw [-latex](-2.5,0)--(-2.6,0);
\draw [-latex](0, 2.5)--(0,2.4);
\draw [-latex](0, -2.5)--(0,-2.4);
\draw [-latex](1.35,1.18)--(1.4,1.13);
\draw [-latex](1.35,-1.18)--(1.4,-1.13);
\draw [-latex](-1.35,-1.18)--(-1.4,-1.13);
\draw [-latex](-1.35,1.18)--(-1.4,1.13);
\node [left] at (-1.8,-0.4)  {$-R_{\infty}$};
\node [right] at (1.8,-0.4)  {$R_{\infty}$};
\node [below] at (0.6,-1.8)  {$-iR_{\infty}$};
\node [above] at (0.6,1.8)  {$iR_{\infty}$};
\node at (1.8,1.8)  {$\widetilde{D}^{+}$};
\node at (1.8,-1.8)  {$\widetilde{D}^{-}$};
\node at (-1.8,-1.8)  {$\widetilde{D}^{+}$};
\node at (-1.8,1.8)  {$\widetilde{D}^{-}$};
\node at (0.8,0.8)  {$\widetilde{D}^{-}$};
\node at (0.8,-0.8)  {$\widetilde{D}^{+}$};
\node at (-0.8,0.8)  {$\widetilde{D}^{+}$};
\node at (-0.8,-0.8)  {$\widetilde{D}^{-}$};
\draw(0,0)circle(1.8cm);
\end {tikzpicture}
\end{center}
\caption{Jump contour $\Gamma$ and analytic domains $\widetilde{D}^{\pm}$ for $N(x,k)$.}
\label{f2}
\end{figure}

Combining the above analysis, we introduce a new piecewise analytic matrix-valued function
\begin{equation}\label{ahkf}
N(x,k)=\left\{
\begin{aligned}
&n(x,k),\quad k\in\mathbb{C} \setminus\left(B(0,R_{\infty})\cup\Gamma_{0}\cup\Gamma_{\infty}\right),\\
&n_{2}(x,k),\quad k\in B(0,R_{\infty})\setminus\Gamma_{0},
\end{aligned}
\right.
\end{equation}
where $\Gamma_{\infty}=\left\{k\in\mathbb{C}:|k|=R_{\infty}\right\}$. This function satisfies \eqref{psix} and tends to the identity matrix $I$ as  $x\rightarrow +\infty$.  The jump condition for $N(x,k)$ is given by
\begin{equation}\label{juv}
N_{+}(x,k)=N_{-}(x,k)\,e^{\frac{i}{2}k^{2}x\widehat{\sigma}_{3}}\,R(k),\quad k\in\Gamma,
\end{equation}
where the jump contour $\Gamma=\Gamma_{0}\cup\Gamma_{\infty}$ and the analytic domains $\widetilde{D}^{\pm}$ are characterized in Figure \ref{f2}. More specifically,
for $k\in\Gamma_{1}=\Gamma_{0}\cap(\mathbb{C}\setminus B(0,R_{\infty}))$,
  \begin{equation}\label{v1}
  R_{1}(k)=\begin{pmatrix}1&-\overline{r(\overline{k})}\\-r(k)&1+r(k)\overline{r(\overline{k})}\end{pmatrix}.
  \end{equation}
For $k\in\Gamma_{2}=\Gamma_{\infty}\cap D^{+}$, it is
sufficient to consider $x=x_{0}$, then we have
  \begin{equation}\label{v2}
R_{2}(k)=e^{-\frac{i}{2}k^{2}x_{0}\widehat{\sigma}_{3}}\left(n_{0}(x_{0},k)\right)^{-1}m(x_{0},k)=\begin{pmatrix}1&\frac{\psi_{12}^{-}(x_{0},k)e^{-ik^{2}x_{0}}}{d(k)d_{0}(k)}\\0&1\end{pmatrix}.
\end{equation}
For $k\in\Gamma_{3}=\Gamma_{0}\cap B(0,R_{\infty})$,
  \begin{equation}\label{v3}
  R_{3}(k)=\begin{pmatrix}1+r_{0}(k)\overline{r_{0}(\overline{k})}&\overline{r_{0}(\overline{k})}\\r_{0}(k)&1\end{pmatrix}.
  \end{equation}
For $k\in\Gamma_{4}=\Gamma_{\infty}\cap D^{-}$,
\begin{equation}\label{v4}
R_{4}(k)=e^{-\frac{i}{2}k^{2}x_{0}\widehat{\sigma}_{3}}m^{-1}(x_{0},k)n_{0}(x_{0},k)=\begin{pmatrix}1&0\\-\frac{\psi_{21}^{-}(x_{0},k)e^{ik^{2}x_{0}}}{a(k)a_{0}(k)}&1\end{pmatrix}.
\end{equation}

To address the subsequent inverse scattering estimate for $x\leq0$, we define an auxiliary matrix-valued function
\begin{equation}\label{au1}
\widetilde{n}(x,k)=\left\{
\begin{aligned}
&\left[\frac{\psi_{1}^{+}(x,k)}{d(k)},\psi_{2}^{-}(x,k)\right],\quad k\in D^{+},\\
&\left[\psi_{1}^{-}(x,k),\frac{\psi_{2}^{+}(x,k)}{a(k)}\right],\quad k\in D^{-},
\end{aligned}
\right.
\end{equation}
which constructs a left-normalized $\widetilde{N}(x,k)$ approaching $I$ as $x\rightarrow -\infty$. Actually, the relationship between $n(x,k)$ and $\widetilde{n}(x,k)$ can be established by an auxiliary matrix $\gamma(k)$
\begin{equation}\label{au2}
\widetilde{n}(x,k)=n(x,k)\gamma(k),
\end{equation}
where
\begin{equation*}
\gamma(k)=\left\{
\begin{aligned}
&\begin{pmatrix}\frac{1}{d(k)}&0\\0&d(k)\end{pmatrix},\quad k\in D^{+}, \\
&\begin{pmatrix}a(k)&0\\0&\frac{1}{a(k)}\end{pmatrix},\quad k\in D^{-}.
\end{aligned}
\right.
\end{equation*}
Denoting by $\widetilde{v}(k)$ the jump matrix associated with $\widetilde{n}(x,k)$, we have the relation $\widetilde{v}(k)=\gamma_{-}^{-1}(k)v(k)\gamma_{+}(k)$. This allows us to restrict attention to $N(x,k)$.

\subsection{ZS spectral problem on the $z$-plane}\label{sub23}
Let us recall the following transformation introduced in \cite{Chengglo}
\begin{equation}\label{trans}
\Psi(x,z)=A(x,k)\psi(x,k)B(k),
\end{equation}
with $z=k^{2}$ and
\begin{equation*}
A(x,k)=\begin{pmatrix}1&0\\-\bar{u}_{x}&-ik\end{pmatrix},\quad B(k)=\begin{pmatrix}1&0\\0&ik^{-1}\end{pmatrix}.
\end{equation*}
The KN spectral problem \eqref{psix} is thereby transformed into a ZS type spectral problem
\begin{equation}\label{Ps}
\Psi_{x}=\frac{i}{2}z\left[\sigma_{3},\Psi\right]+\widetilde{U}\Psi,
\end{equation}
where
\begin{equation*}
\widetilde{U}=\begin{pmatrix}i|u_{x}|^{2}&iu_{x}\\-\bar{u}_{xx}-i\bar{u}_{x}|u_{x}|^{2}&-i|u_{x}|^{2}\end{pmatrix}.
\end{equation*}
In particularly, for $z\in\mathbb{C}\setminus\{0\}$, we define
\begin{equation}
\begin{aligned}
\Psi^{\pm}(x,z):=&\,A(x,k)\psi^{\pm}(x,k)B(k)\\
=&\begin{pmatrix}\psi_{11}^{\pm}&ik^{-1}\psi_{12}^{\pm}\\
-\bar{u}_{x}\psi_{11}^{\pm}-ik\psi_{21}^{\pm}&-ik^{-1}\bar{u}_{x}\psi_{12}^{\pm}+\psi_{22}^{\pm}
\end{pmatrix}.
\end{aligned}
\end{equation}
When $z=0$, define $\Psi^{\pm}(x,0):=\begin{pmatrix}1&0\\-\bar{u}_{x}&1\end{pmatrix}$. Denote $\Psi^{\pm}=\left[\Psi_{1}^{\pm},\Psi_{2}^{\pm}\right]$, then $\Psi_{1}^{+}(x,z)$ and $\Psi_{2}^{-}(x,z)$ are analytic for $z\in \mathbb{C}^{+}$, while $\Psi_{1}^{-}(x,z)$ and $\Psi_{2}^{+}(x,z)$ are analytic for $z\in \mathbb{C}^{-}$. From the scattering relation \eqref{sca} on the $k$-plane, we have
\begin{equation}\label{rela}
\Psi^{-}(x,z)=\Psi^{+}(x,z)e^{\frac{i}{2}zx\operatorname{ad}\sigma_{3}}S(z),\quad z\in\mathbb{R},
\end{equation}
where
\begin{equation}
S(z)=\begin{pmatrix}a(z)&\breve{b}(z)\\ \breve{c}(z)&d(z)\end{pmatrix}=\begin{pmatrix}a(k)&ik^{-1}b(k)\\ -ikc(k)&d(k)\end{pmatrix}.
\end{equation}
Here, $\breve{b}(z)$ has no singularity at $z=0$, see \cite{Chengglo}.

Transformation \eqref{trans} not only converts the KN spectral problem into a ZS spectral problem that is more advantageous for scattering analysis, but also significantly simplifies the verification of self-intersection matching conditions. It is worth noting that the matrix function $B(k)$ in this transformation depends exclusively on the variable $k$, enabling us to flexibly construct the necessary transformations for the RH problems. More regularity properties of the eigenfunctions $\Psi^{\pm}(x,z)$ and scattering data defined on the $z$-plane can be found in our previous work \cite{Chengglo}.

\section{Inverse scattering transform}\label{sec3}
\subsection{The classical matrix RH problem}\label{sub31}

Let us consider the reconstruction relationship between the matrix function $N(x,k)$ constructed in \eqref{ahkf} and the potential $u$.  Substituting the following expansion
\begin{equation}\label{exm}
N(x,k)=N_{\infty}(x)+\frac{N_{-1}(x)}{k}+\frac{N_{-2}(x)}{k^{2}}+O(k^{-3}),\quad k\rightarrow\infty
\end{equation}
into \eqref{psix}, we have
\begin{equation}\label{asy}
N(x,k)\rightarrow N_{\infty}(x)=e^{-i\int_{x}^{\infty}|u_{y}(y)|^{2}dy\sigma_{3}},\quad k\rightarrow\infty,
\end{equation}
and
\begin{equation}\label{rei}
u_{x}(x)=-iN_{-1}^{(12)}(x)e^{-i\int_{x}^{\infty}|N_{-1}^{(12)}(y)|^{2}dy},
\end{equation}
where $N_{-1}^{(12)}(x)=\lim\limits_{k\rightarrow\infty}\left(kN(x,k)\right)_{12}$. Here, ``12'' denotes the (1,2)-entry of the matrix function.

Let $m(x,k)=e^{i\int_{x}^{\infty}|u_{y}(y)|^{2}dy\sigma_{3}}\psi(x,k)$, then from \eqref{psix} we have
\begin{equation}\label{Phix}
m_{x}=\frac{i}{2}k^{2}\left[\sigma_{3},m\right]+\left(kU_{1}+U_{2}\right)m,
\end{equation}
where
\begin{equation*}
U_{1}=e^{i\int_{x}^{\infty}|u_{y}(y)|^{2}dy\widehat{\sigma}_{3}}U_{x},\quad U_{2}=-i|u_{x}|^{2}\sigma_{3}.
\end{equation*}
In order to normalize $N(x,k)$, we define a new matrix function
\begin{equation}\label{mm}
M(x,k):=N_{\infty}^{-1}(x)N(x,k),
\end{equation}
then $M(x,k)$ satisfies \eqref{Phix} and
\begin{equation}
M(x,k)\rightarrow I, \quad k\rightarrow \infty.
\end{equation}
From \eqref{juv}, the jump condition for $M(x,k)$ is given by
\begin{equation}\label{matn}
M_{+}(x,k)=M_{-}(x,k)\,e^{\frac{i}{2}k^{2}x\widehat{\sigma}_{3}}\,R(k),\quad k\in\Gamma.
\end{equation}
Therefore, we establish the following RH problem.
\begin{rhp}\label{RH1}
Find a matrix-valued function $M(x,k)$ satisfying:
\begin{itemize}
  \item Analyticity: $M(x,\cdot)$ is analytic in $\mathbb{C}\setminus\Gamma$.
  \item Jump condition:
  On the oriented contour $\Gamma$ (see Figure \ref{f2}), $M(x,k)$ satisfies the jump condition
      \begin{equation}
      M_{+}(x,k)=M_{-}(x,k)R_{x}(k),\quad k\in\Gamma,
      \end{equation}
      where $R_{x}(k)=e^{\frac{i}{2}k^{2}x\widehat{\sigma}_{3}}R(k)$ and $R(k)$ is given by \eqref{v1}--\eqref{v4}.
  \item Normalization: $M(x,k)\rightarrow I$ as $k\rightarrow\infty$.
\end{itemize}
\end{rhp}
It follows from \eqref{rei} and \eqref{mm} that
\begin{equation}\label{recn1}
u_{x}(x)=-iM_{-1}^{(12)}(x)e^{-2i\int_{x}^{\infty}|M_{-1}^{(12)}(y)|^{2}dy},
\end{equation}
where
\begin{equation}\label{m-1}
M_{-1}^{(12)}(x)=\lim\limits_{k\rightarrow\infty}\left(kM(x,k)\right)_{12}.
\end{equation}
For a given oriented contour $\Gamma$ and any function $f\in L^{2}(\Gamma)$, the Cauchy operator is defined by \cite{Zhou1989b,Trogdon}
\begin{equation}
\mathcal{C}_{\Gamma}f(k)=\frac{1}{2\pi i}\int_{\Gamma}\frac{f(\xi)}{\xi-k}d\xi, \quad k\in\mathbb{C}\setminus\Gamma.
\end{equation}
The function $\mathcal{C}_{\Gamma}f$ is analytic off the contour $\Gamma$. The projection operators are given by
\begin{equation}
\mathcal{C}_{\Gamma}^{\pm}f(k)=\lim_{\substack{k'\to k \\ k'\in \widetilde{D}^{\pm}}}\frac{1}{2\pi i}\int_{\Gamma}\frac{f(\xi)}{\xi-k'}d\xi, \quad k\in \Gamma,
\end{equation}
where the limits are taken nontangentially. The operator $\mathcal{C}_{\Gamma}^{\pm}$ is bounded from $L^{2}(\Gamma)$ to $L^{2}(\Gamma)$. By classical inverse scattering theory, the RH problem \ref{RH1} is equivalent to the Beals-Coifman equation:
\begin{equation}
\mu(x,k)=I+\left(\mathcal{C}_{R_{x}}\mu\right)(x,k),
\end{equation}
where $\mu=M_{+}R_{x+}^{-1}=M_{-}R_{x-}^{-1}$ with $R_{x}=R_{x-}^{-1}R_{x+}$ and
\begin{equation}\label{opv}
\mathcal{C}_{R_{x}}\mu=\mathcal{C}_{\Gamma}^{+}\mu(I-R_{x-})+\mathcal{C}_{\Gamma}^{-}\mu(R_{x+}-I).
\end{equation}
A key feature of the RH Problem \ref{RH1} on the $k$-plane is that its jump matrix $R(k)$ satisfies Zhou's vanishing lemma \cite{Zhou1989b}, as demonstrated in the following lemma.
\begin{lemma}\label{vanish}
The jump matrix $R(k)$ given by \eqref{v1}--\eqref{v4} satisfies the following properties:
\begin{enumerate}[label=(\roman*)]
  \item For $k\in\mathbb{R}$, $R(k)+R^{\dag}(k)$ is positive definite.
  \item For $k \in \Gamma \setminus \mathbb{R}$, $R(k)=R^{\dag}(\overline{k})$.
\end{enumerate}
\end{lemma}
\begin{proof}
This lemma follows from direct calculation. Clearly,
\begin{equation*}
R_{1}(k)+R_{1}^{\dag}(k)=2\begin{pmatrix}1&-\overline{r(k)}\\-r(k)&1+|r(k)|^{2}\end{pmatrix},\quad k\in\mathbb{R}\cap\Gamma_{1},
\end{equation*}
and
\begin{equation*}
R_{3}(k)+R_{3}^{\dag}(k)=2\begin{pmatrix}1+|r_{0}(k)|^{2}&\overline{r_{0}(k)}\, \\r_{0}(k)&1\end{pmatrix},\quad k\in\mathbb{R}\cap\Gamma_{3},
\end{equation*}
is positive definite. Applying the symmetry constraints $d(k)=\overline{a(\overline{k})}$ and $\psi_{21}^{-}(k)=-\overline{\psi_{12}^{-}(\overline{k})}$, property (ii) is immediate.
\end{proof}

\begin{figure}
\begin{center}
\begin{tikzpicture}
\draw [fill=pink,ultra thick,color=yellow!10](-3.5,0) rectangle (3.5,3);
 \draw [fill=pink,ultra thick,color=white!10] (-3.5,0) rectangle (3.5,-2.2);
\filldraw[color=white!10](0,0)-- (2,0) arc (0:180:2);
\filldraw[color=yellow!10](0,0)-- (2,0) arc (0:-180:2);
\draw(0,0)circle(2cm);
\draw[fill,black] (2,0) circle [radius=0.03];
\draw[fill,black] (-2,0) circle [radius=0.03];
\draw [ ](-3.5,0)--(3.5,0)  node[right, scale=1] {$\mathbb{R}$};
\draw [-latex](2.95,0)--(3,0);
\draw [-latex](-2.65,0)--(-2.6,0);
\draw [-latex](0,0)--(-0.1,0);
\draw [-latex]  (0.15,2)--(0.2,2);
\draw [-latex]  (0.15,-2)--(0.2,-2);
\node [right] at (-2,-0.4)  {$-R_{\infty}^{2}$};
\node [right] at (2,-0.4)  {$R_{\infty}^{2}$};
\node at (2.3,1.7)  {$\Omega^{+}$};
\node at (2.3,-1.7)  {$\Omega^{-}$};
\node at (0,1)  {$\Omega^{-}$};
\node at (0,-1)  {$\Omega^{+}$};
\end{tikzpicture}
\end{center}
\caption{ Jump contour $\Lambda$ and analytic domains $\Omega^{\pm}$ for $P(x,z)$.}
\label{f3}
\end{figure}
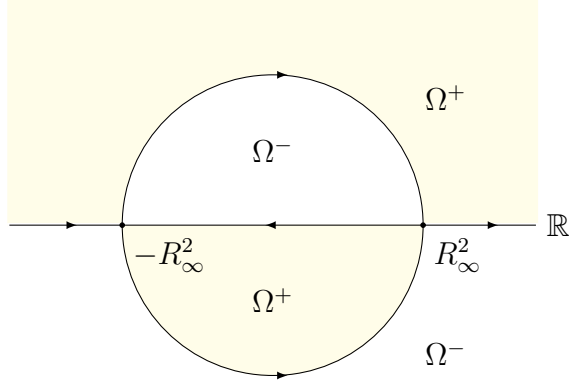

\subsection{Constructing Type I vector RH problems}\label{sub32}
Based on the transformation \eqref{trans}, we now consider a piecewise analytic row-vector-valued function on the $z$-plane
\begin{equation}\label{map}
P(x,z):=\left[A(x,k)M(x,k)B(k)\right]^{(R_{1})},
\end{equation}
where the superscript ``$(R_{1})$'' denotes the first row of the matrix. Note that
\begin{equation}\label{que}
A(x,k)M(x,k)B(k)\rightarrow\begin{pmatrix}1&0\\
-\bar{u}_{x}+\bar{u}_{x}e^{-2i\int_{x}^{\infty}|u_{y}|^{2}dy}&1\end{pmatrix},\quad z\rightarrow\infty,
\end{equation}
that is,
\begin{equation}
P(x,z)\rightarrow (1,0),\quad z\rightarrow\infty,
\end{equation}
which is why we adopt the row-vector-valued function. From \eqref{matn}, we can derive the jump matrix of $P(x,z)$
\begin{equation}
P_{+}(x,z)=P_{-}(x,z)e^{\frac{i}{2}zx\widehat{\sigma}_{3}}G(z),\quad z\in\Lambda,
\end{equation}
where
\begin{equation}\label{rg}
G(z)=B^{-1}(k)R(k)B(k).
\end{equation}
The jump contour
$\Lambda=\mathbb{R}\cup\Lambda_{\infty}$  with  $\Lambda_{\infty}=\{z\in\mathbb{C}:|z|=R_{\infty}^{2}\}$ and the positive and negative regions of $\Lambda$ are shown in Figure \ref{f3}. By using \eqref{v1}--\eqref{v4}, we are able to obtain the detailed expression for $G(z)$. For $z\in\Lambda_{1}=(-\infty,-R_{\infty}^{2})\cap(R_{\infty}^{2},+\infty)$,
       \begin{equation}\label{j1}
       G_{1}(z)=\begin{pmatrix} 1&r_{1}(z)\\-r_{2}(z)&1-r_{1}(z)r_{2}(z)
       \end{pmatrix},
       \end{equation}
 where
      \begin{equation}
      r_{1}(z)=-ik^{-1}\overline{r(\overline{k})}\quad \text{and} \quad r_{2}(z)=-ikr(k).
      \end{equation}
Clearly,  $r_{2}(z)=-z\overline{r_{1}(z)}$.
For $\Lambda_{2}=\Lambda_{\infty}\cap\mathbb{C}^{+}$,
       \begin{equation}\label{j2}
       G_{2}(z)=\begin{pmatrix}1&\frac{\Psi_{12}^{-}(x_{0},z)e^{-izx_{0}}}{d(z)d_{0}(z)}\\0&1
       \end{pmatrix}.
       \end{equation}
For $\Lambda_{3}=(R_{\infty}^{2},-R_{\infty}^{2})$,
       \begin{equation}\label{j3}
       G_{3}(z)=\begin{pmatrix}1-r_{01}(z)r_{02}(z)&-r_{01}(z)\\r_{02}(z)&1
       \end{pmatrix},
       \end{equation}
       where $r_{01}(z)=-ik^{-1}\overline{r_{0}(\overline{k})}$ and $r_{02}(z)=-ikr_{0}(k)$.
For $\Lambda_{4}=\Lambda_{\infty}\cap \mathbb{C}^{-}$,
       \begin{equation}\label{j4}
       G_{4}(z)=\begin{pmatrix}1&0\\ \frac{z\overline{\Psi_{12}^{-}(x_{0},\overline{z})}e^{izx_{0}}}{a(z)a_{0}(z)}&1
        \end{pmatrix}.
       \end{equation}
By \eqref{map}, we rewrite the reconstruction formulas \eqref{recn1} as follows
\begin{equation}\label{recp1}
u_{x}(x)=-p(x)e^{-2i\int_{x}^{\infty}|p(y)|^{2}dy},
\end{equation}
where
\begin{equation}\label{px}
p(x)=\lim\limits_{z\rightarrow\infty}\left(zP\left(x,z\right)\right)_{12}
\end{equation}
and the subscript ``12'' denotes the second component of the row vector. Building upon these results, we construct the Type I row vector RH problem as follows:
\begin{rhp}\label{RH2}
Find a row-vector-valued function $P(x,z)$ satisfying:
\begin{itemize}
  \item Analyticity: $P(x,\cdot)$ is analytic in $\mathbb{C}\setminus\Lambda$.
  \item Jump condition: $P(x,z)$ has continuous boundary values $P_{\pm}(x,z)$ as $z'\rightarrow z$ from $\Omega^{\pm}$ to $\Lambda$ (see Figure \ref{f3}) and
      \begin{equation}
      P_{+}(x,z)=P_{-}(x,z)G_{x}(z),\quad z\in\Lambda,
      \end{equation}
      where $G_{x}(z)=e^{\frac{i}{2}zx\widehat{\sigma}_{3}}G(z)$ and $G(z)$ is given by \eqref{j1}--\eqref{j4}.
  \item Normalization: $P(x,z)\rightarrow (1,0)$ as $z\rightarrow\infty$.
\end{itemize}
\end{rhp}
We present the Beals-Coifman integral equation equivalent to the RH Problem \ref{RH2} as follows:
\begin{equation}\label{nu}
\omega_{p}(x,z)=(1,0)+\left(\mathcal{C}_{G_{x}}\omega_{p}\right)(x,z),
\end{equation}
where $\omega_{p}=P_{+}G_{x+}^{-1}=P_{-}G_{x-}^{-1}$ with $G_{x}=G_{x-}^{-1}G_{x+}$  and
\begin{equation}
\mathcal{C}_{G_{x}}\omega_{p}=\mathcal{C}_{\Lambda}^{+}\omega_{p}(I-G_{x-})+\mathcal{C}_{\Lambda}^{-}\omega_{p}(G_{x+}-I).
\end{equation}
Unlike conventional configurations where scattering data resides on the real axis \cite{Chengglo}, our framework defines these data on a self-intersecting contour $\Lambda$ (see Figure \ref{f3}). This distinction necessitates explicit analysis of the regularity properties of jump matrix $G(z)$ on $\Lambda$ for inverse scattering estimation.
\begin{proposition}\label{vuyf}
Let $u\in H^{3}(\mathbb{R})\cap H^{2,1}(\mathbb{R})$. Then the jump matrix $G(z)$ given by \eqref{j1}--\eqref{j4} admits a triangular factorization
\begin{equation}
G(z)=G_{-}^{-1}(z)G_{+}(z), \quad z\in\Lambda,
\end{equation}
where
$G_{\pm}-I$ and $G_{\pm}^{-1}-I\in H^{1}(\partial\Omega^{\pm})\cap L^{2,1}(\partial\Omega^{\pm})$.
\end{proposition}
\begin{proof}
From our prior results in \cite{Chengglo} (Propositions 3.5, 3.9 and 3.14), we use parameterization to derive
\begin{equation*}
G-I,\ G^{-1}-I\in H^{1}(\Lambda)\cap L^{2,1}(\Lambda).
\end{equation*}
Using the expression for $G_{2}$ in \eqref{j2} and the foundational relation \eqref{rela}, we obtain that $G_2(z)$ admits the following decomposition at the critical points  $z=\pm R_{\infty}^{2}$
\begin{equation*}
\begin{aligned}
G_{2}(\pm R_{\infty}^{2})
&=\begin{pmatrix}1&e^{-i(\pm R_{\infty}^{2})x_{0}}\frac{\Psi_{12}^{-}(x_{0},\pm R_{\infty}^{2})}{d(\pm R_{\infty}^{2})d_{0}(\pm R_{\infty}^{2})}\\0&1\end{pmatrix}\\
&=\begin{pmatrix}1&\frac{\Psi_{11}^{+}(x_{0},\pm R_{\infty}^{2})\breve{b}(\pm R_{\infty}^{2})}{d(\pm R_{\infty}^{2})d_{0}(\pm R_{\infty}^{2})}\\0&1\end{pmatrix}\begin{pmatrix}1&e^{-i(\pm R_{\infty}^{2})x_{0}}\frac{\Psi_{12}^{+}(x_{0},\pm R_{\infty}^{2})}{d_{0}(\pm R_{\infty}^{2})}\\0&1\end{pmatrix}\\
&=\begin{pmatrix}1&-r_{01}(\pm R_{\infty}^{2})\\0&1\end{pmatrix}\begin{pmatrix}1&r_{1}(\pm R_{\infty}^{2})\\0&1\end{pmatrix}.
\end{aligned}
\end{equation*}
 Similarly, the matrix $G_4$ satisfies the decomposition
\begin{equation*}
\begin{aligned}
G_{4}(\pm R_{\infty}^{2})
&=\begin{pmatrix}1&0\\e^{i(\pm R_{\infty}^{2})x_{0}}\frac{\pm R_{\infty}^{2}\overline{\Psi_{12}^{-}(x_{0},\pm R_{\infty}^{2})}}{a(\pm R_{\infty}^{2})a_{0}(\pm R_{\infty}^{2})}&1\end{pmatrix}\\
&=\begin{pmatrix}1&0\\ \frac{\pm R_{\infty}^{2}\overline{\Psi_{11}^{+}(x_{0},\pm R_{\infty}^{2})}\,\overline{\breve{b}(\pm R_{\infty}^{2})}}{\overline{d(\pm R_{\infty}^{2})}\,\overline{d_{0}(\pm R_{\infty}^{2})}}&1\end{pmatrix}\begin{pmatrix}1&0\\ e^{i(\pm R_{\infty}^{2})x_{0}} \frac{\pm R_{\infty}^{2}\overline{\Psi_{12}^{+}(x_{0},\pm R_{\infty}^{2})}}{\overline{d_{0}(\pm R_{\infty}^{2})}}&1\end{pmatrix}\\
&=\begin{pmatrix}1&0\\-r_{2}(\pm R_{\infty}^{2})&1\end{pmatrix}\begin{pmatrix}1&0\\r_{02}(\pm R_{\infty}^{2})&1\end{pmatrix}.
\end{aligned}
\end{equation*}
Combining \eqref{j1}, \eqref{j3} and the above decompositions, we obtain
\begin{equation*}
G_{1}^{-1}(\pm R_{\infty}^{2})G_{4}(\pm R_{\infty}^{2})=G_{2}^{-1}(\pm R_{\infty}^{2})G_{3}(\pm R_{\infty}^{2}),
\end{equation*}
which implies
\begin{equation*}
G_{1}G_{2}^{-1}G_{3}G_{4}^{-1}=I+o(1), \quad \text{as} \quad k\rightarrow +R_{\infty}^{2}
\end{equation*}
and
\begin{equation*}
G_{3}^{-1}G_{2}G_{1}^{-1}G_{4}=I+o(1), \quad \text{as} \quad k\rightarrow -R_{\infty}^{2}.
\end{equation*}
Applying \cite{Zhou1999} and Theorem 2.56 in \cite{Trogdon}, we conclude that the jump matrix $G(z)$ defined on $\Lambda$ satisfies matching conditions, thereby proving the proposition. Here, it is important to emphasize that $G_{1+}$, $G_{2\pm}$ and $G_{3-}$ are upper triangular, while $G_{1-}$, $G_{3+}$ and $G_{4\pm}$ are lower triangular.
\end{proof}

Although $G_{\pm}(z)-I$ satisfies the matching condition at the self-intersection points $\pm R_{\infty}^{2}$, the lack of decay behavior compels us to transform $P(x,z)$ for advancing inverse scattering estimates. By \cite{Zhou1998,LiY}, there exists a matrix function $\eta\in \mathcal{A}(\mathbb{C}\setminus\Lambda)$ satisfying the following properties for $l=1$:
\begin{enumerate}[label=(\roman*)]
  \item $\eta_{\pm}\in \mathcal{R}(\partial\Omega^{\pm})$ and $\eta_{\pm}-I=O(z^{-2})$ as $z\rightarrow \infty$.
  \item $\eta_{\pm}$ shares the triangular structure of $G_{\pm}$.
  \item $\eta_{\pm}(z)=G_{\pm}(z)+o\left((z-a)^{l-1}\right)$ for $a=\pm R_{\infty}^{2}$.
\end{enumerate}
Here, $\mathcal{A}(\mathbb{C}\setminus\Lambda)$ is the analytic function space, and $\mathcal{R}(\partial\Omega^{\pm})$ contains functions with rational restrictions on boundary components.  For $x\geq0$,  let us define a new function
\begin{equation}
P^{(1)}(x,z):=P(x,z)e^{\frac{i}{2}zx\widehat{\sigma}_{3}}\eta^{-1}(z),
\end{equation}
then $P^{(1)}(x,z)$ satisfies the following RH problem.
\begin{rhp}\label{RHp1}
Find a row-vector-valued function $P^{(1)}(x,z)$ satisfying:
\begin{itemize}
  \item Analyticity: $P^{(1)}(x,\cdot)$ is analytic in $\mathbb{C}\setminus\Lambda$.
  \item Jump condition: $P^{(1)}(x,z)$ has continuous boundary values $P^{(1)}_{\pm}(x,z)$ as $z'\rightarrow z$ from $\Omega^{\pm}$ to $\Lambda$ (see Figure \ref{f3}) and
      \begin{equation}
      P^{(1)}_{+}(x,z)=P^{(1)}_{-}(x,z)G^{(1)}_{x}(z),\quad z\in\Lambda,
      \end{equation}
      where $G^{(1)}_{x}(z)=e^{\frac{i}{2}zx\widehat{\sigma}_{3}}G^{(1)}(z)$ with
      \begin{equation}
      G^{(1)}(z)=\eta_{-}(z)G(z)\eta_{+}^{-1}(z).
      \end{equation}
  \item Normalization: $P^{(1)}(x,z)\rightarrow (1,0)$ as $z\rightarrow\infty$.
\end{itemize}
\end{rhp}
Let $\omega_{p}^{(1)}$ denote the solution of the Beals-Coifman integral equation associated with RH problem \ref{RHp1}, then it can be verified that $\omega_{p}^{(1)}=\omega_{p}$. Decomposing $G^{(1)}$ via $G^{(1)}=(G^{(1)}_{-})^{-1}G^{(1)}_{+}$, where $G^{(1)}_{-}=G_{-}\eta_{-}^{-1}$ and $G^{(1)}_{+}=G_{+}\eta_{+}^{-1}$. The properties of $\eta_{\pm}$ imply that $G^{(1)}_{\pm}-I$ belong to $H^{1}(\partial\Omega^{\pm})\cap L^{2,1}(\partial\Omega^{\pm})$ and vanish at $\pm R_{\infty}^{2}$. For $1\leq j\leq4$, the upper and lower triangular properties of the matrix function $G^{(1)}_{j\pm}$ are identical to those of $G_{j\pm}$.

\begin{figure}
\begin{center}
\begin{tikzpicture}
\draw [fill=pink,ultra thick,color=yellow!10](-3.5,0) rectangle (3.5,3);
 \draw [fill=pink,ultra thick,color=white!10] (-3.5,0) rectangle (3.5,-2.2);
\filldraw[color=white!10](0,0)-- (2,0) arc (0:180:2);
\filldraw[color=yellow!10](0,0)-- (2,0) arc (0:-180:2);
\filldraw [color=yellow!10] (-2,0) -- plot [domain=-2:2,smooth] (\x,{sqrt(-\x*\x/4+1)}) -- (2,0) -- (0,0);
\filldraw [color=white!10] (-2,0) --(2,0) plot [domain=-2:2,smooth] (\x,{-sqrt(-\x*\x/4+1)});
\draw(0,0)circle(2cm);
\draw[rotate around={0:(0,0)}] (0,0) ellipse (2 and 1);
\draw[fill,black] (2,0) circle [radius=0.03];
\draw[fill,black] (-2,0) circle [radius=0.03];
\draw [ ](-3.5,0)--(3.5,0)  node[right, scale=1] {$\mathbb{R}$};
\draw [-latex](2.95,0)--(3,0);
\draw [-latex](-2.65,0)--(-2.6,0);
\draw [-latex](0,0)--(0.1,0);
\draw [-latex]  (0.15,2)--(0.2,2);
\draw [-latex]  (0.15,-2)--(0.2,-2);
\draw [-latex]  (-0.1,1)--(-0.15,1);
\draw [-latex]  (-0.1,-1)--(-0.15,-1);
\node at (2.5, 1.6 )  {$\widetilde{\Omega}^{+}$};
\node at (2.5, -1.6 )  {$\widetilde{\Omega}^{-}$};
\node at (-0.5, 0.5 )  {$\widetilde{\Omega}^{+}$};
\node at (0.5, -0.5 )  {$\widetilde{\Omega}^{-}$};
\node at (0.5, 1.5 )  {$\widetilde{\Omega}^{-}$};
\node at (-0.5, -1.5 )  {$\widetilde{\Omega}^{+}$};
\node [right] at (-2,-0.3)  {$-R_{\infty}^{2}$};
\node [right] at (2,-0.3)  {$R_{\infty}^{2}$};
\end{tikzpicture}
\end{center}
\caption{Jump contour $\widetilde{\Lambda}$ and analytic domains $\widetilde{\Omega}^{\pm}$ for $P^{(2)}(x,z)$.}
\label{f4}
 \end{figure}
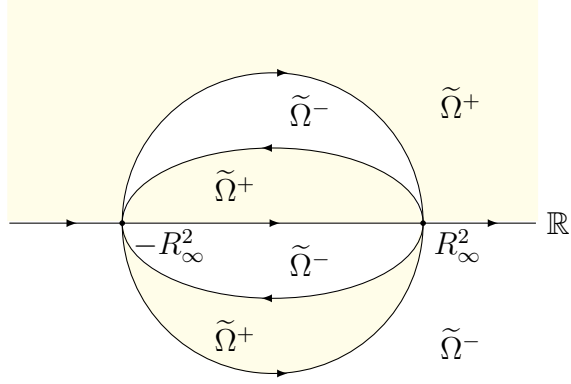

To apply the Cauchy projection operator on $\mathbb{R}$ and unify the triangular properties of $G^{(1)}_{\pm}$, we introduce a new jump contour $\widetilde{\Lambda}$ (see Figure \ref{f4}). The contour $\widetilde{\Lambda}$ contains two extra elliptic arcs not present in $\Lambda$: $\widetilde{\Lambda}_{5}$ in the upper half-plane and $\widetilde{\Lambda}_{6}$ in the lower half-plane. It is specially noted that \(\widetilde{\Lambda}_{3}=\left(-R_{\infty}^{2},R_{\infty}^{2}\right)\), and it is oriented in the opposite direction to $\Lambda_{3}$. For the remaining components, we denote $\widetilde{\Lambda}_{j}=\Lambda_{j}$ with $j=1,2,4$.  Define $G^{(2)}(z)$ on $\widetilde{\Lambda}$ by:
For $j=1,2,4$,
\begin{equation}\label{g21}
G_{j}^{(2)}(z)=G_{j}^{(1)}(z),\quad z\in\widetilde{\Lambda}_{j}.
\end{equation}
For $j=3$,
\begin{equation}\label{g22}
G_{3}^{(2)}(z)=(G_{3}^{(1)}(z))^{-1},\quad z\in\widetilde{\Lambda}_{3}.
\end{equation}
For $j=5,6$,
\begin{equation}\label{g23}
G_{j}^{(2)}(z)=I,\quad z\in\widetilde{\Lambda}_{j}.
\end{equation}
This function admits the decomposition $G^{(2)}=(G^{(2)}_{-})^{-1}G^{(2)}_{+}$ with
\begin{equation}
G_{2-}^{(2)}(z)=G_{4+}^{(2)}(z)=I, \quad G_{3\pm}^{(2)}(z)=G_{3\mp}^{(1)}(z).
\end{equation}
This leads to the following RH problem:
\begin{rhp}\label{RHp2}
Find a row-vector-valued function $P^{(2)}(x,z)$ satisfying:
\begin{itemize}
  \item Analyticity: $P^{(2)}(x,\cdot)$ is analytic in $\mathbb{C}\setminus\widetilde{\Lambda}$.
  \item Jump condition: $P^{(2)}(x,z)$ has continuous boundary values $P^{(2)}_{\pm}(x,z)$ as $z'\rightarrow z$ from $\widetilde{\Omega}^{\pm}$ to $\widetilde{\Lambda}$ (see Figure \ref{f4}) and
      \begin{equation}
      P^{(2)}_{+}(x,z)=P^{(2)}_{-}(x,z)G^{(2)}_{x}(z),\quad z\in\widetilde{\Lambda},
      \end{equation}
      where $G^{(2)}_{x}(z)=e^{\frac{i}{2}zx\widehat{\sigma}_{3}}G^{(2)}(z)$ and $G^{(2)}(z)$  is given by \eqref{g21}--\eqref{g23}.
  \item Normalization: $P^{(2)}(x,z)\rightarrow (1,0)$ as $z\rightarrow\infty$.
\end{itemize}
\end{rhp}

The Beals-Coifman integral equation for RH problem \ref{RHp2} is given by
\begin{equation}
\omega_{p}^{(2)}(x,z)=(1,0)+(\mathcal{C}_{G_{x}^{(2)}}\omega_{p}^{(2)})(x,z),
\end{equation}
where
\begin{equation}
\mathcal{C}_{G_{x}^{(2)}}\omega_{p}^{(2)}=\mathcal{C}_{\widetilde{\Lambda}}^{+}\omega_{p}^{(2)}(I-G_{x-}^{(2)})+\mathcal{C}_{\widetilde{\Lambda}}^{-}\omega_{p}^{(2)}(G_{x+}^{(2)}-I).
\end{equation}
Furthermore, the solution $\omega_{p}^{(2)}$  satisfies the equality $\omega_{p}^{(2)}(z)=\omega_{p}^{(1)}(z)=\omega_{p}(z)$ for $z\in\widetilde{\Lambda}_{j}, 1\leq j\leq4$. $G^{(2)}_{\pm}$ retains the properties of  $G^{(1)}_{\pm}$, that is,  $G^{(2)}_{\pm}-I$ belong to $H^{1}(\partial\widetilde{\Omega}^{\pm})\cap L^{2,1}(\partial\widetilde{\Omega}^{\pm})$ and vanish at $\pm R_{\infty}^{2}$. Additionally, $\forall 1\leq j\leq6$, $G^{(2)}_{j-}$  is strictly lower-triangular and $G^{(2)}_{j+}$  is strictly upper-triangular.  Therefore, we will use $G^{(2)}(z)$ and the reconstruction formula \eqref{recp1} to estimate $u_{x}$.

\subsection{Constructing Type II vector RH problems}\label{sub33}
Within the inverse scattering framework, the reconstruction formula \eqref{recp1} play a central role in potential estimation.
However, according to Zhou's $L^{2}$-bijectivity theory \cite{Zhou1998}, the reconstruction formula \eqref{recp1} and the condition $G^{(2)}_{\pm}-I\in H^{1}(\partial\widetilde{\Omega}^{\pm})\cap L^{2,1}(\partial\widetilde{\Omega}^{\pm})$ yield only $u_{x}\in H^{1}(\mathbb{R})\cap L^{2,1}(\mathbb{R})$. To achieve the enhanced regularity $u\in H^{3}(\mathbb{R})\cap H^{2,1}(\mathbb{R})$, supplementary reconstruction formulas must be developed.

We now introduce another row-vector-valued function $H(x,z)$, defined by
\begin{equation}\label{mah}
H(x,z):=\left[N_{\infty}^{-1}(x)A(x,k)N(x,k)B(k)\right]^{(R_{2})},
\end{equation}
where the superscript ``$(R_{2})$'' denotes the second row of the matrix. Distinct from \eqref{map}, \eqref{mah} performs transformation to the $z$-plane prior to normalization. This improved method overcomes the limitations in \eqref{que}, ensuring the asymptotic behavior
\begin{equation}
 H(x,z)\rightarrow (0,1),\quad z\rightarrow\infty.
 \end{equation}
Furthermore, it establishes a novel reconstruction formula:
\begin{equation}\label{sfs}
h(x)=i\,\partial_{x}\left(\bar{u}_{x}e^{-i\int_{x}^{\infty}|u_{y}|^{2}dy}\right)e^{-i\int_{x}^{\infty}|u_{y}|^{2}dy}=\lim\limits_{z\rightarrow\infty}\left(zH\left(x,z\right)\right)_{11},
\end{equation}
where the subscript ``11'' denotes the first component of the row vector. The validity of this reconstruction formula is guaranteed by a key equality established in Proposition 3.6 of our previous work \cite{Chengglo}:
\begin{equation}\label{asjk}
i\,\partial_{x}\left(\bar{u}_{x}e^{-i\int_{x}^{\infty}|u_{y}|^{2}dy}\right)=\lim_{z\rightarrow\infty}z\Psi_{21}^{+}(x,z).
\end{equation}
It is apparent from the \eqref{juv} that
\begin{equation}
H_{+}(x,z)=H_{-}(x,z)e^{\frac{i}{2}zx\widehat{\sigma}_{3}}G(z),\quad z\in\Lambda,
\end{equation}
where $G(z)$ are given by \eqref{j1}--\eqref{j4}. Therefore, we construct the Type II row vector RH problem:
\begin{rhp}\label{RH3}
Find a row-vector-valued function $H(x,z)$ satisfying:
\begin{itemize}
  \item Analyticity: $H(x,\cdot)$ is analytic in $\mathbb{C}\setminus\Lambda$.
  \item Jump condition:
  $H(x,z)$ has continuous boundary values $H_{\pm}(x,z)$ as $z'\rightarrow z$ from $\Omega^{\pm}$ to $\Lambda$ (see Figure \ref{f3}) and
      \begin{equation}
      H_{+}(x,z)=H_{-}(x,z)G_{x}(z),\quad z\in\Lambda,
      \end{equation}
      where $G_{x}(z)=e^{\frac{i}{2}zx\widehat{\sigma}_{3}}G(z)$ and $G(z)$ is given by \eqref{j1}--\eqref{j4}.
  \item Normalization: $H(x,z)\rightarrow (0,1)$ as $z\rightarrow\infty$.
\end{itemize}
\end{rhp}
The Beals-Coifman integral equation associated with RH Problem \ref{RH3} is as follows:
\begin{equation}\label{abh}
\omega_{h}(x,z)=(0,1)+\left(\mathcal{C}_{G_{x}}\omega_{h}\right)(x,z),
\end{equation}
where $\omega_{h}=H_{+}G_{x+}^{-1}=H_{-}G_{x-}^{-1}$. Analogously, define the transformed function
\begin{equation}
H^{(1)}(x,z):=H(x,z)e^{\frac{i}{2}zx\widehat{\sigma}_{3}}\eta^{-1}(z),
\end{equation}
and consider the $G^{(2)}(z)$  given by \eqref{g21}--\eqref{g23}. Then, we successively construct the following two RH problems:
\begin{rhp}\label{RHh1}
Find a row-vector-valued function $H^{(1)}(x,z)$ satisfying:
\begin{itemize}
  \item Analyticity: $H^{(1)}(x,\cdot)$ is analytic in $\mathbb{C}\setminus\Lambda$.
  \item Jump condition: $H^{(1)}(x,z)$ has continuous boundary values $H^{(1)}_{\pm}(x,z)$ as $z'\rightarrow z$ from $\Omega^{\pm}$ to $\Lambda$ (see Figure \ref{f3}) and
      \begin{equation}
      H^{(1)}_{+}(x,z)=H^{(1)}_{-}(x,z)G^{(1)}_{x}(z),\quad z\in\Lambda,
      \end{equation}
      where $G^{(1)}_{x}(z)=e^{\frac{i}{2}zx\widehat{\sigma}_{3}}G^{(1)}(z)$ with
      \begin{equation}
      G^{(1)}(z)=\eta_{-}(z)G(z)\eta_{+}^{-1}(z).
      \end{equation}
  \item Normalization: $H^{(1)}(x,z)\rightarrow (0,1)$ as $z\rightarrow\infty$.
\end{itemize}
\end{rhp}
\begin{rhp}\label{RHh2}
Find a row-vector-valued function $H^{(2)}(x,z)$ satisfying:
\begin{itemize}
  \item Analyticity: $H^{(2)}(x,\cdot)$ is analytic in $\mathbb{C}\setminus\widetilde{\Lambda}$.
  \item Jump condition: $H^{(2)}(x,z)$ has continuous boundary values $H^{(2)}_{\pm}(x,z)$ as $z'\rightarrow z$ from $\widetilde{\Omega}^{\pm}$ to $\widetilde{\Lambda}$ (see Figure \ref{f4}) and
      \begin{equation}
      H^{(2)}_{+}(x,z)=H^{(2)}_{-}(x,z)G^{(2)}_{x}(z),\quad z\in\widetilde{\Lambda},
      \end{equation}
      where $G^{(2)}_{x}(z)=e^{\frac{i}{2}zx\widehat{\sigma}_{3}}G^{(2)}(z)$ and $G^{(2)}(z)$  is given by \eqref{g21}--\eqref{g23}.
  \item Normalization: $H^{(2)}(x,z)\rightarrow (0,1)$ as $z\rightarrow\infty$.
\end{itemize}
\end{rhp}
Denote by $\omega_{h}^{(1)}$ and $\omega_{h}^{(2)}$ the solutions of the Beals-Coifman integral equations for the RH problems \ref{RHh1} and \ref{RHh2}, respectively. We will then employ $G^{(2)}(z)$ and the reconstruction formula \eqref{sfs} to derive estimates for $u_{xx}$.

\subsection{Time evolution of the jump matrix}\label{sub34}
The classical inverse scattering transform method usually first ignores time $t$, considers only the Lax pair \eqref{psix} related to $x$, and then presents the time evolution. Before analyzing the solvability of the RH problems, we first perform time evolution of the jump matrix. Given the fundamental solutions $N_{\pm}(x,t,k)$ of the Lax pair \eqref{psix}--\eqref{psit} that satisfy the jump condition
\begin{equation*}
N_{+}(x,t,k)=N_{-}(x,t,k)\,e^{\frac{i}{2}k^{2}x\widehat{\sigma}_{3}}\,R(t,k).
\end{equation*}
Substituting $N_{-}(x,t,k)\,e^{\frac{i}{2}k^{2}x\sigma_{3}}\,R(t,k)$ into \eqref{psit} and using the asymptotic condition $N_{\pm}(x,t,k)\rightarrow I$, as $x\rightarrow +\infty$, we obtain
\begin{equation*}
R_{t}(t,k)=\frac{i}{2}k^{-2}[\sigma_{3},R(t,k)].
\end{equation*}
Integration then yields
\begin{equation*}
R(t,k)=e^{\frac{i}{2}k^{-2}t\widehat{\sigma}_{3}}R(0,k).
\end{equation*}
Using the relation between $G$ and $R$ in \eqref{rg} and the fact that  $B(k)$ depends only on the variable $k$, we have
\begin{equation}
G(t,z)=e^{\frac{i}{2}z^{-1}t\widehat{\sigma}_{3}}G(0,z).
\end{equation}
Clearly,
\begin{equation}
G_{\pm}(t,z)=e^{\frac{i}{2}z^{-1}t\widehat{\sigma}_{3}}G_{\pm}(0,z).
\end{equation}

From Propositions 3.10 and 3.14 in our previous work \cite{Chengglo}, we can derive $z^{-1}(G_{\pm}-I )\in H^{1}(\partial\Omega^{\pm})$. According to Proposition \ref{vuyf}, our previous results \cite{Chengglo}, and the relation between $G_{\pm}^{(2)}$ and $G_{\pm}$, we obtain the following proposition.
\begin{proposition}\label{timep}
For initial data $u_{0}\in  H^{3}(\mathbb{R})\cap H^{2,1}(\mathbb{R})$, we have $G^{(2)}_{\pm}(t,\cdot)-I \in H^{1}(\partial\widetilde{\Omega}^{\pm})\cap L^{2,1}(\partial\widetilde{\Omega}^{\pm})$ and $z^{-1}(G^{(2)}_{\pm}(t,\cdot)-I )\in H^{1}(\partial\widetilde{\Omega}^{\pm})$,  $\forall\, t\in [-T,T]$ with $T>0$. Moreover, the map $u_{0}\rightarrow G^{(2)}_{\pm}-I$ is Lipschitz continuous from  $H^{3}(\mathbb{R})\cap H^{2,1}(\mathbb{R})$ to $C([-T,T], H^{1}(\partial\widetilde{\Omega}^{\pm})\cap L^{2,1}(\partial\widetilde{\Omega}^{\pm}))$ for every $T>0$.
\end{proposition}
\subsection{Unique solvability of the RH problems}\label{sub35}
This subsection focuses on the solvability of the above RH problems for the given jump matrix $G_{\pm}(z)$ satisfying $G_{\pm}-I \in H^{1}(\partial\Omega^{\pm})\cap L^{2,1}(\partial\Omega^{\pm})$.
It follows from equations \eqref{nu} and \eqref{abh} that the solvability of both RH Problems \ref{RH2} and \ref{RH3} is equivalent to the invertibility of the operator $I-\mathcal{C}_{G_{x}}$. Moreover, the solvability of other RH problems can be directly derived from these two RH problems. The relationships between these RH problems are summarized in Figure \ref{ft}.

\begin{figure}
 \centering
\begin{tikzpicture}
[node distance=1cm and 1.3cm,
   arrow/.style={->}]
\node (N) {$N$};
\node[right=of N] (M) {$M$};
\node[right=of M] (P) {$P=[AMB]^{(R_1)}$};
\node[right=of P] (P1) {$P^{(1)}$};
\node[right=of P1] (P2) {$P^{(2)}$};
\node[below=of P] (H) {$H=[N_{\infty}^{-1}ANB]^{(R_2)}$};
\node[right=of H] (H1) {$H^{(1)}$};
\node[right=of H1] (H2) {$H^{(2)}$};
\draw[arrow] (N) -- (M)node [above=0.15em, midway] {$N_{\infty}^{-1}$};
\draw[arrow] (M) -- (P);
\draw[arrow] (P) -- (P1)node [above=0.15em, midway] {$\eta^{-1}$};
\draw[arrow] (P1) -- (P2)node [above=0.15em, midway] {$G^{(2)}$};
\draw[arrow] (H) -- (H1)node [above=0.15em, midway] {$\eta^{-1}$};
\draw[arrow] (H1) -- (H2)node [above=0.15em, midway] {$G^{(2)}$};
\draw[<->] (P) -- (H) node [right=0.4em, midway] {$I-\mathcal{C}_{G_{x}}$};
\draw[->] (0,-0.3) -- (0,-1.75) -- (3.3,-1.75);
\end{tikzpicture}
\caption{The relationships among RH problems.}
\label{ft}
\end{figure}

However, the jump matrix $G(z)$ constructed on the $z$-plane via \eqref{j1}--\eqref{j4} fails to satisfy Zhou's vanishing lemma. Fortunately, the jump matrix $R(k)$ on the $k$-plane satisfies this condition; see Lemma \ref{vanish}. We therefore establish a connection between
the operators $I-\mathcal{C}_{G_{x}}$ and $I-\mathcal{C}_{R_{x}}$, as detailed in the following lemma. This connection plays an important role in proving the solvability of the RH problems. In fact, the solvability of the RH problems can be guaranteed under the weaker condition that $G_{\pm}-I\in H^{\frac{1}{2}+\varepsilon}(\partial\Omega^{\pm})$ for $\varepsilon>0$. This relaxation also ensures the Lipschitz continuity of the subsequent inverse scattering transform.

\begin{lemma}\label{munu}
Let $G_{\pm}-I\in H^{\frac{1}{2}+\varepsilon}(\partial\Omega^{\pm})$ for $\varepsilon>0$. Assume that for all $x\in\mathbb{R}$,  $\tilde{\omega}(x,\cdot)=(\tilde{\omega}_{11},\tilde{\omega}_{12})(x,\cdot)$ is the solution of the integral equation in $L^{2}(\Lambda)$
\begin{equation}\label{guu}
\left(I-\mathcal{C}_{G_{x}}\right)\tilde{\omega}=0.
\end{equation}
Define the matrix function
\begin{equation}\label{tilde}
\tilde{\mu}(x,k)=\begin{pmatrix}\tilde{\mu}_{11}(x,k)&\tilde{\mu}_{12}(x,k)\\
\tilde{\mu}_{21}(x,k)&\tilde{\mu}_{22}(x,k)\end{pmatrix}=\begin{pmatrix}\tilde{\omega}_{11}(x,k^{2})&-ik\tilde{\omega}_{12}(x,k^{2})\\
-ik\overline{\tilde{\omega}_{12}(x,\overline{k}^{2})}&\overline{\tilde{\omega}_{11}(x,\overline{k}^{2})}\end{pmatrix},
\end{equation}
then
$\tilde{\mu}(x,\cdot)\in L^{2}(\Gamma)$ solves
\begin{equation}\label{gus}
\left(I-\mathcal{C}_{R_{x}}\right)\tilde{\mu}=0.
\end{equation}
\end{lemma}
\begin{proof}
Selecting $k\in\Gamma_{1}$ as a representative case, we provide a detailed proof here. By the definition of operator $\mathcal{C}_{R_{x}}$ in \eqref{opv}, the fact that $\tilde{\mu}(x,k)$ solves \eqref{gus} implies
\begin{equation}\label{mu11}
\begin{aligned}
\tilde{\mu}_{11}(x,k)=&-\mathcal{C}_{\Gamma_{1}}^{+}(\tilde{\mu}_{12}(x,k)r(k)e^{-ik^{2}x})(k)+\int_{\Gamma_{3}}\frac{\tilde{\mu}_{12}(x,\xi)r_{0}(\xi)e^{-i\xi^{2}x}}{\xi-k}\frac{d\xi}{2\pi i}\\
&-\int_{\Gamma_{4}}\frac{\tilde{\mu}_{12}(x,\xi)\psi_{21}^{-}(x_{0},\xi)e^{-i\xi^{2}(x-x_{0})}}{(\xi-k)a(\xi)a_{0}(\xi)}\frac{d\xi}{2\pi i},\\
\end{aligned}
\end{equation}
\begin{equation}\label{mu12}
\begin{aligned}
\tilde{\mu}_{12}(x,k)=&\,-\mathcal{C}_{\Gamma_{1}}^{-}(\tilde{\mu}_{11}(x,k)\overline{r(\overline{k})}e^{ik^{2}x})(k)+\int_{\Gamma_{3}}\frac{\tilde{\mu}_{11}(x,\xi)\overline{r_{0}(\overline{\xi})}e^{i\xi^{2}x}}{\xi-k}\frac{d\xi}{2\pi i}\\
&+\int_{\Gamma_{2}}\frac{\tilde{\mu}_{11}(x,\xi)\psi_{12}^{-}(x_{0},\xi)e^{i\xi^{2}(x-x_{0})}}{(\xi-k)d(\xi)d_{0}(\xi)}\frac{d\xi}{2\pi i},\\
\end{aligned}
\end{equation}
\begin{equation}\label{mu21}
\begin{aligned}
\tilde{\mu}_{21}(x,k)=&-\mathcal{C}_{\Gamma_{1}}^{+}(\tilde{\mu}_{22}(x,k)r(k)e^{-ik^{2}x})(k)+\int_{\Gamma_{3}}\frac{\tilde{\mu}_{22}(x,\xi)r_{0}(\xi)e^{-i\xi^{2}x}}{\xi-k}\frac{d\xi}{2\pi i}\\
&-\int_{\Gamma_{4}}\frac{\tilde{\mu}_{22}(x,\xi)\psi_{21}^{-}(x_{0},\xi)e^{-i\xi^{2}(x-x_{0})}}{(\xi-k)a(\xi)a_{0}(\xi)}\frac{d\xi}{2\pi i},\\
\end{aligned}
\end{equation}
\begin{equation}\label{mu22}
\begin{aligned}
\tilde{\mu}_{22}(x,k)=&\,-\mathcal{C}_{\Gamma_{1}}^{-}(\tilde{\mu}_{21}(x,k)\overline{r(\overline{k})}e^{ik^{2}x})(k)+\int_{\Gamma_{3}}\frac{\tilde{\mu}_{21}(x,\xi)\overline{r_{0}(\overline{\xi})}e^{i\xi^{2}x}}{\xi-k}\frac{d\xi}{2\pi i}\\
&+\int_{\Gamma_{2}}\frac{\tilde{\mu}_{21}(x,\xi)\psi_{12}^{-}(x_{0},\xi)e^{i\xi^{2}(x-x_{0})}}{(\xi-k)d(\xi)d_{0}(\xi)}\frac{d\xi}{2\pi i}.\\
\end{aligned}
\end{equation}
Since $\tilde{\omega}(x,z)$ solves \eqref{guu}, for $z=k^{2}\in\Lambda_{1}$ we obtain
\begin{equation}\label{nu11}
\begin{aligned}
\tilde{\omega}_{11}(x,z)=&-\mathcal{C}_{\Lambda_{1}}^{+}\left(\tilde{\omega}_{12}(x,z)r_{2}(z)e^{-izx}\right)(z)+\int_{\Lambda_{3}}\frac{\tilde{\omega}_{12}(x,\rho)r_{02}(\rho)e^{-i\rho x}}{\rho-z}\frac{d\rho}{2\pi i}\nonumber\\
&+\int_{\Lambda_{4}}\frac{\tilde{\omega}_{12}(x,\rho)\,\rho\, \overline{\Psi_{12}^{-}(x_{0},\overline{\rho})} e^{-i\rho (x-x_{0})}}{(\rho-z)a(\rho)a_{0}(\rho)}\frac{d\rho}{2\pi i},\\
\end{aligned}
\end{equation}
\begin{equation}\label{nu12}
\begin{aligned}
\tilde{\omega}_{12}(x,z)=&\,\mathcal{C}_{\Lambda_{1}}^{-}\left(\tilde{\omega}_{11}(x,z)r_{1}(z)e^{izx}\right)(z)-\int_{\Lambda_{3}}\frac{\tilde{\omega}_{11}(x,\rho)r_{01}(\rho)e^{i\rho x}}{\rho-z}\frac{d\rho}{2\pi i}\\
&+\int_{\Lambda_{2}}\frac{\tilde{\omega}_{11}(x,\rho) \Psi_{12}^{-}(x_{0},\rho) e^{i\rho (x-x_{0})}}{(\rho-z)d(\rho)d_{0}(\rho)}\frac{d\rho}{2\pi i},
\end{aligned}
\end{equation}
where $\rho=\xi^{2}$. Through the spectral parameterization, we can derive the contour integral correspondence
\begin{equation}\label{1tra}
\int_{\Gamma_{j}}f(k) dk=\int_{\Lambda_{j}}\frac{ f(\sqrt{z})-f(-\sqrt{z})}{2\sqrt{z}}dz, \quad 1\leq j\leq4.
\end{equation}
As demonstrated in \cite{Liu16}, if $f$ is an even function on $\Gamma_{j}$, then
\begin{equation}\label{2even}
\mathcal{C}_{\Gamma_{j}}^{\pm}(f)(k)=\mathcal{C}_{\Lambda_{j}}^{\pm}(g)(z),\quad 1\leq j\leq4,
\end{equation}
where $g(z)=f(\sqrt{z})$.
If $f$ is an odd function on $\Gamma_{j}$, then
\begin{equation}\label{3odd}
\mathcal{C}_{\Gamma_{j}}^{\pm}(f)(k)=k\, \mathcal{C}_{\Lambda_{j}}^{\pm}(h)(z),\quad 1\leq j\leq4,
\end{equation}
where $h(z)=f(\sqrt{z})/\sqrt{z}$.

Building upon the above framework, we first prove that the function $\tilde{\mu}$ defined by \eqref{tilde} belongs to $L^{2}(\Gamma)$. Specifically, it suffices to demonstrate that $\tilde{\mu}(I-R_{x-})$ and  $\tilde{\mu}(R_{x+}-I)\in L^{2}(\Gamma)$. As an illustrative example, consider proving $\tilde{\mu}_{12}r\in L^{2}(\Gamma_{1})$. Since $G_{\pm}-I\in H^{\frac{1}{2}+\varepsilon}(\partial\Omega^{\pm})$, by using \eqref{tilde} and parameterization, we obtain
\begin{equation*}
\begin{aligned}
\int_{\Gamma_{1}}|\tilde{\mu}_{12}(x,k)r(k)|^{2}|dk|&=\int_{\Lambda_{1}}|\tilde{\omega}_{12}(x,z)r_{1}(z)|^{2}|z|^{\frac{3}{2}}dz\\
&\leq \|r_{2}\|_{L^{\infty}(\Lambda_{1})}\int_{\Lambda_{1}}|\tilde{\omega}_{12}(x,z)|^{2}|r_{1}(z)||z|^{\frac{1}{2}}dz\\
&\leq\|r_{1}\|_{L^{\infty}(\Lambda_{1})}^{\frac{1}{2}}\|r_{2}\|_{L^{\infty}(\Lambda_{1})}^{\frac{3}{2}}\|\tilde{\omega}_{12}(x,\cdot)\|_{L^{2}(\Lambda_{1})}^{2}.
\end{aligned}
\end{equation*}

Next, we verify the $\tilde{\omega}_{11}(x,z)$ given by $\eqref{nu11}$ satisfies \eqref{mu11}. Since both $\tilde{\mu}_{12}$ and $r$ are odd functions with respect to $k$, by combining \eqref{tilde} and \eqref{2even}, we obtain
\begin{equation*}
-\mathcal{C}_{\Lambda_{1}}^{+}\left(\tilde{\omega}_{12}(x,z)r_{2}(z)e^{-izx}\right)(z)
=-\mathcal{C}_{\Gamma_{1}}^{+}(\tilde{\mu}_{12}(x,k)r(k)e^{-ik^{2}x})(k).
\end{equation*}
Furthermore, it follows from \eqref{tilde} and \eqref{1tra} that
\begin{equation*}
\begin{aligned}
&\int_{\Lambda_{3}}\frac{\tilde{\omega}_{12}(x,\rho)r_{02}(\rho)e^{-i\rho x}}{\rho-z}\frac{d\rho}{2\pi i}
=\int_{\Lambda_{3}}\frac{\tilde{\mu}_{12}(x,\sqrt{\rho})r_{0}(\sqrt{\rho})e^{-i\rho x}}{\rho-k^{2}}\frac{d\rho }{2\pi i}\\
=&\,\frac{1}{2\pi i}\int_{\Lambda_{3}}\frac{\tilde{\mu}_{12}(x,\sqrt{\rho})r_{0}(\sqrt{\rho})e^{-i\rho x}}{\sqrt{\rho}-k}-\frac{\tilde{\mu}_{12}(x,-\sqrt{\rho})r_{0}(-\sqrt{\rho})e^{-i\rho x}}{-\sqrt{\rho}-k}\frac{d\rho }{2\sqrt{\rho}}\\
=&\int_{\Gamma_{3}}\frac{\tilde{\mu}_{12}(x,\xi)r_{0}(\xi)e^{-i\xi^{2} x}}{\xi-k}\frac{d\xi}{2\pi i}
\end{aligned}
\end{equation*}
and the third term of $\tilde{\omega}_{11}(x,z)$ equals
\begin{equation*}
\begin{aligned}
\int_{\Lambda_{4}}\frac{\tilde{\omega}_{12}(x,\rho)\,\rho\, \overline{\Psi_{12}^{-}(x_{0},\overline{\rho})} e^{-i\rho (x-x_{0})}}{(\rho-z)a(\rho)a_{0}(\rho)}\frac{d\rho}{2\pi i}
=&-\int_{\Lambda_{4}}\frac{\tilde{\mu}_{12}(x,\sqrt{\rho})\psi_{21}^{-}(x_{0},\sqrt{\rho}) e^{-i\rho (x-x_{0})}}{(\rho-k^{2})a(\sqrt{\rho})a_{0}(\sqrt{\rho})}\frac{d\rho}{2\pi i}\\
=&-\int_{\Gamma_{4}}\frac{\tilde{\mu}_{12}(x,\xi)\psi_{21}^{-}(x_{0},\xi) e^{-i\xi^{2} (x-x_{0})}}{(\xi-k)a(\xi)a_{0}(\xi)}\frac{d\xi}{2\pi i}.
\end{aligned}
\end{equation*}
Then, we verify the $\overline{\tilde{\omega}_{11}(x,\overline{k}^{2})}$ satisfy \eqref{mu22}. From \eqref{tilde} and \eqref{2even}, we obtain
\begin{equation*}
\begin{aligned}
-\overline{\mathcal{C}_{\Lambda_{1}}^{+}\left(\tilde{\omega}_{12}(x,z)r_{2}(z)e^{-izx}\right)(z)}
=&\,\mathcal{C}_{\Lambda_{1}}^{-}\left(\overline{\tilde{\omega}_{12}(x,z)}\,\overline{r_{2}(z)}e^{izx}\right)(z)\\
=&\,-\mathcal{C}_{\Gamma_{1}}^{-}(\tilde{\mu}_{21}(x,k)\overline{r(\overline{k})}e^{ik^{2}x})(k),
\end{aligned}
\end{equation*}
where we have used the identity $\overline{\mathcal{C}_{\Lambda_{1}}^{+}(f)}=-\mathcal{C}_{\Lambda_{1}}^{-}(\overline{f})$. Similarly, we can calculate that
\begin{equation*}
\begin{aligned}
\int_{\Lambda_{3}}\frac{\overline{\tilde{\omega}_{12}(x,\rho)} \,\overline{r_{02}(\rho)}e^{i\rho x}}{\rho-z}\frac{d\rho}{2\pi i}
=-\int_{\Gamma_{3}}\frac{\tilde{\mu}_{21}(x,\xi) \overline{r_{0}(\overline{\xi})}e^{i\xi^{2} x}}{\xi-k}\frac{d\xi}{2\pi i}
\end{aligned}
\end{equation*}
by using \eqref{tilde} and \eqref{1tra}. Let $g(z)=\int_{\Lambda_{4}}\frac{f(\rho)}{\rho-z}\frac{d\rho}{2\pi i}$. By parameterization, we derive
\begin{equation*}
\overline{g(\overline{z})}=-\int_{\Lambda_{2}}\frac{\overline{f(\overline{\rho})}}{\rho-z}\frac{d\rho}{2\pi i},
\end{equation*}
which implies
\begin{equation*}
\begin{aligned}
\text{the third term of }\overline{\tilde{\omega}_{11}(x,\overline{k}^{2})}
=&-\int_{\Lambda_{2}}\frac{\overline{\tilde{\omega}_{12}(x,\overline{\rho})}\, \rho \, \Psi_{12}^{-}(x_{0},\rho) e^{i\rho (x-x_{0})}}{(\rho-z)d(\rho)d_{0}(\rho)}\frac{d\rho}{2\pi i}\\
=&\int_{\Gamma_{2}}\frac{\tilde{\mu}_{21}(x,\xi) \psi_{12}^{-}(x_{0},\xi) e^{i\xi^{2} (x-x_{0})}}{(\xi-k)d(\xi)d_{0}(\xi)}\frac{d\xi}{2\pi i},
\end{aligned}
\end{equation*}
where the final identity follows from \eqref{tilde} and \eqref{1tra}. Applying the same analytical approach, we can verify that $-ik\tilde{\omega}_{12}(x,k^{2})$ and $-ik\overline{\tilde{\omega}_{12}(x,\overline{k}^{2})}$ satisfy expressions \eqref{mu12} and \eqref{mu21} respectively. Furthermore, when $k\in \Gamma_{j}$ for $2\leq j\leq4$, the conclusions remain valid via similar arguments.
\end{proof}

\begin{proposition}\label{con}
Assume that $G_{\pm}-I\in H^{\frac{1}{2}+\varepsilon}(\partial\Omega^{\pm})$ for $\varepsilon>0$. Then RH Problems \ref{RH2} and \ref{RH3} are uniquely solvable.
\end{proposition}
\begin{proof}
The proof reduces to demonstrating unique solvability of the Beals-Coifman equation \eqref{nu} and \eqref{abh}. For $G_{\pm}-I\in H^{\frac{1}{2}+\varepsilon}(\partial\Omega^{\pm})$, the operator  $I-\mathcal{C}_{G_{x}}$ is a Fredholm operator with zero index, as $G_{x+}-I$ and $I-G_{x-}$ admit uniform rational approximations.
If $\tilde{\omega}\in \ker_{L^{2}(\Lambda)}(I-\mathcal{C}_{G_{x}})$,  then Lemma \ref{munu} allows us to construct a matrix $\tilde{\mu}\in \ker_{L^{2}(\Gamma)}(I-\mathcal{C}_{R_{x}})$. Combining Lemma \ref{vanish} and Zhou's vanishing lemma in \cite{Zhou1989b}, we have $\ker_{L^{2}(\Gamma)}(I-\mathcal{C}_{R_{x}})=0$,  hence $\tilde{\mu}=0$. By reversing the relation \eqref{tilde}, we conclude $\tilde{\omega}=0$. The Fredholm theory implies that
$I-\mathcal{C}_{G_{x}}$ is invertible in $L^{2}(\Lambda)$. Moreover, the unique solutions to the RH problems \ref{RH2} and \ref{RH3} can be respectively expressed as
\begin{equation}\label{ppp}
P(x,z)=(1,0)+\frac{1}{2\pi i}\int_{\Lambda}\frac{\omega_{p}(x,\rho)\left(G_{x+}(\rho)-G_{x-}(\rho)\right)}{\rho-z}d\rho,
\end{equation}
\begin{equation}\label{hhh}
     H(x,z)=(0,1)+\frac{1}{2\pi i}\int_{\Lambda}\frac{\omega_{h}(x,\rho)\left(G_{x+}(\rho)-G_{x-}(\rho)\right)}{\rho-z}d\rho.
      \end{equation}
\end{proof}
\section{Reconstruction and estimation of the potential}\label{sec4}
\subsection{Reconstruction formulas}\label{sub41}
We have constructed the formulas \eqref{recp1} and  \eqref{sfs} at $z\to\infty$ to reconstruct $u_{x}$ and $u_{xx}$. To recover the original potential $u$, we now need to analyse the asymptotic expansion of $N(x,k)$ as $k\to 0$ by employing the $t$-Lax pair \eqref{psit}.

Substituting
\begin{equation}
N(x,k)=N_{0}(x)+kN_{1}(x)+k^{2}N_{2}(x)+O(k^{3}),\quad k\rightarrow0
\end{equation}
into \eqref{psix} and \eqref{psit}, we can derive
\begin{equation}\label{rez}
u(x)=\lim_{k\rightarrow0}\left(k^{-1}N(x,k)\right)_{12}.
\end{equation}
From \eqref{mm} and \eqref{map}, the  formula \eqref{rez} can be restructured as
\begin{equation}\label{recp2}
\begin{aligned}
u(x)=&\,\lim_{k\rightarrow0}\left(k^{-1}M(x,k)\right)_{12}e^{-i\int_{x}^{\infty}|M_{-1}^{(12)}(y)|^{2}dy}\\
=&-i\lim_{z\rightarrow0}\left(P(x,z)\right)_{12}e^{-i\int_{x}^{\infty}|p(y)|^{2}dy},
\end{aligned}
\end{equation}
where $M_{-1}^{(12)}(x)$ and $p(x)$ are given by \eqref{m-1} and \eqref{px}, respectively.

Reconstruction formulas \eqref{recp1}, \eqref{sfs} and \eqref{recp2} indicate that we need only to estimate $p(x)$, $h(x)$ and $u(x)$. Here, the $p(x)$ and $h(x)$ are given by \eqref{px} and \eqref{sfs}, respectively. Using the expressions for $P(x,z)$ and $H(x,z)$ in \eqref{ppp}--\eqref{hhh}, we obtain
\begin{align}
p(x)= & \left(-\frac{1}{2\pi i} \int_{\Lambda}\omega_{p}(x,z)\left(G_{x+}(z)-G_{x-}(z)\right)dz\right)_{12},\label{rpp}\\
h(x)= &\left(-\frac{1}{2\pi i} \int_{\Lambda}\omega_{h}(x,z)\left(G_{x+}(z)-G_{x-}(z)\right)dz\right)_{11},\label{rhh}\\
u(x)= &-\frac{e^{-i\int_{x}^{\infty}|p(y)|^{2}dy}}{2\pi} \left(\int_{\Lambda}\omega_{p}(x,z)\left(z^{-1}G_{x+}(z)-z^{-1}G_{x-}(z)\right)dz\right)_{12}.\label{ruu}
\end{align}
Through transformations between RH problems and relations of the corresponding Beals-Coifman solutions, the reconstruction formulas \eqref{rpp}--\eqref{ruu} can be rewritten as
\begin{align}
&p(x)= \left(-\frac{1}{2\pi i} \int_{\widetilde{\Lambda}}\omega^{(2)}(x,z)\left(G^{(2)}_{x+}(z)-G^{(2)}_{x-}(z)\right)dz\right)_{12},\label{rrp}\\
&h(x)= \left(-\frac{1}{2\pi i} \int_{\widetilde{\Lambda}}\omega^{(2)}(x,z)\left(G^{(2)}_{x+}(z)-G^{(2)}_{x-}(z)\right)dz\right)_{21},\label{rrh}\\
&u(x)= -\frac{e^{-i\int_{x}^{\infty}|p(y)|^{2}dy}}{2\pi} \left(\int_{\widetilde{\Lambda}}\omega^{(2)}(x,z)\left(z^{-1}G^{(2)}_{x+}(z)-z^{-1}G^{(2)}_{x-}(z)\right)dz\right)_{12},\label{rru}
\end{align}
where
\begin{equation}
\omega^{(2)}=\begin{pmatrix}\omega_{p}^{(2)}\\ \omega_{h}^{(2)}\end{pmatrix}=(I-\mathcal{C}_{G^{(2)}_{x}})^{-1}I.
\end{equation}
According to previous results \cite{LiY}, we can deduce that
$(I-\mathcal{C}_{G^{(2)}_{x}})^{-1}$ is bounded from  $L^{2}(\widetilde{\Lambda})$ to $L^{2}(\widetilde{\Lambda})$, and the map
\begin{equation}\label{lpxz}
G^{(2)}\mapsto\{x\mapsto (I-\mathcal{C}_{G^{(2)}_{x}})^{-1}\}
\end{equation}
is Lipschitz continuous into the space $C\left([0,\infty);\mathscr{B}\left(L^{2}\right)\right)$.

\subsection{Estimates of the reconstructed potential}\label{sub42}
\begin{figure}
\begin{center}
\begin{tikzpicture}
\draw [ ](-4,0)--(2,0)  node[right, scale=1] {$\mathbb{R}$};
\draw (1,0) arc (0:180:2);
\draw(1,0) arc (0:180:2 and 1);
\draw[fill,black] (1,0) circle [radius=0.03];
\draw[fill,black] (-3,0) circle [radius=0.03];
\node  at (-3,-0.4)  {$-R_{\infty}^{2}$};
\node  at (1,-0.4)  {$R_{\infty}^{2}$};
\draw [-latex](-1,0)--(-0.9,0);
\draw [-latex](1.5,0)--(1.6,0);
\draw [-latex](-3.5,0)--(-3.4,0);
\draw [-latex]  (-1,2)--(-0.9,2);
\draw [-latex]  (-1,1)--(-1.1,1);
\node at (-2.2, 2.1)  {$\widetilde{\Lambda}_{2}$};
\node at (-1.5, 1.3)  {$\widetilde{\Lambda}_{5}$};

\draw [ ](3,0)--(9,0)  node[right, scale=1] {$\mathbb{R}$};
\draw[fill,black] (8,0) circle [radius=0.03];
\draw[fill,black] (4,0) circle [radius=0.03];
\draw (8,0) arc (0:-180:2);
\draw(8,0) arc (0:-180:2 and 1);
\draw [-latex](6,0)--(6.1,0);
\draw [-latex](8.5,0)--(8.6,0);
\draw [-latex](3.5,0)--(3.6,0);
\draw [-latex](6,-1)--(5.9,-1);
\draw [-latex](6,-2)--(6.1,-2);
\node  at (4,0.4)  {$-R_{\infty}^{2}$};
\node  at (8,0.4)  {$R_{\infty}^{2}$};
\node at (7.3, -2)  {$\widetilde{\Lambda}_{4}$};
\node at (6.5, -1.35)  {$\widetilde{\Lambda}_{6}$};
\end{tikzpicture}
\end{center}
\caption{The auxiliary contours $\widetilde{\Lambda}_{+}=(\mathbb{C}^{+}\cap\widetilde{\Lambda})\cup\mathbb{R}$ (left) and $\widetilde{\Lambda}_{-}=(\mathbb{C}^{-}\cap\widetilde{\Lambda})\cup\mathbb{R}$ (right).}
\label{fjia}
 \end{figure}
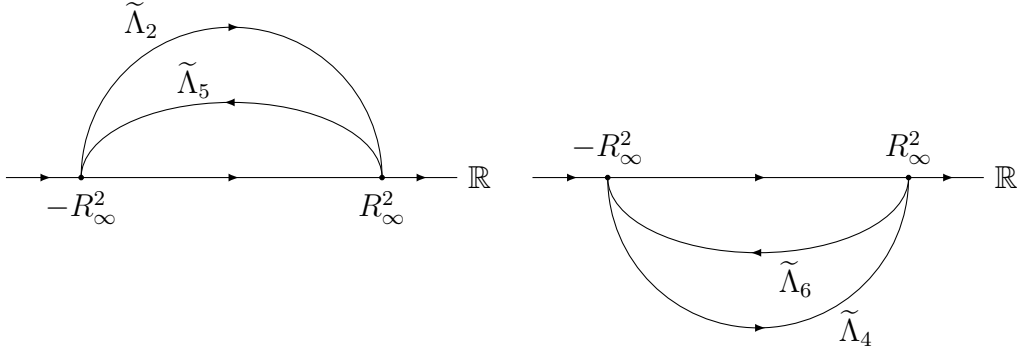

We will use \eqref{rrp}--\eqref{rru} to estimate the reconstruction potential.
Define the auxiliary contours
\begin{equation}
\widetilde{\Lambda}_{\pm}:=(\mathbb{C}^{\pm}\cap\widetilde{\Lambda})\cup\mathbb{R},
\end{equation}
as shown in Figure \ref{fjia}.  Here we note that
\begin{equation}\label{po1}
G^{(2)}_{x+}(z)-I=0,\quad z\in \widetilde{\Lambda}\setminus\widetilde{\Lambda}_{+}
\end{equation}
and
\begin{equation}\label{po2}
G^{(2)}_{x-}(z)-I=0,\quad z\in \widetilde{\Lambda}\setminus\widetilde{\Lambda}_{-}.
\end{equation}
We state the following lemma without proof. A detailed proof can be found in \cite{Zhou1998}.
\begin{lemma}\label{imlem}
Let $G^{(2)}_{\pm}-I\in H^{1}(\partial\widetilde{\Omega}^{\pm})$. Then, for all $x\geq0$, we have
\begin{align}
&\|\mathcal{C}_{\mathbb{R}}^{-}(G^{(2)}_{x+}-I)\|_{L^{2}(\mathbb{R})}\leq\frac{1}{\sqrt{1+x^{2}}}\|G^{(2)}_{+}-I\|_{H^{1}},\label{e1}\\
&\|\mathcal{C}_{\mathbb{R}}^{+}(G^{(2)}_{x-}-I)\|_{L^{2}(\mathbb{R})}\leq\frac{1}{\sqrt{1+x^{2}}}\|G^{(2)}_{-}-I\|_{H^{1}},\label{e2}\\
&\left\|\mathcal{C}_{\widetilde{\Lambda}\rightarrow\widetilde{\Lambda}_{-}}^{-}(G^{(2)}_{x+}-I)\right\|_{L^{2}(\widetilde{\Lambda}_{-})}\leq\frac{c}{\sqrt{1+x^{2}}}\|G^{(2)}_{+}-I\|_{H^{1}},\label{e3}\\
&\left\|\mathcal{C}_{\widetilde{\Lambda}\rightarrow\widetilde{\Lambda}_{+}}^{+}(G^{(2)}_{x-}-I)\right\|_{L^{2}(\widetilde{\Lambda}_{+})}\leq\frac{c}{\sqrt{1+x^{2}}}\|G^{(2)}_{-}-I\|_{H^{1}},\label{e4}\\
&\|G^{(2)}_{x+}-I\|_{L^{2}(\widetilde{\Lambda}_{2})}\leq\frac{c}{\sqrt{1+x^{2}}}\|G^{(2)}_{+}-I\|_{H^{1}},\label{e5}\\
&\|G^{(2)}_{x-}-I\|_{L^{2}(\widetilde{\Lambda}_{4})}\leq\frac{c}{\sqrt{1+x^{2}}}\|G^{(2)}_{-}-I\|_{H^{1}},\label{e6}\\
&\|(\mathcal{C}_{G^{(2)}_{x}})^{2}I\|_{L^{2}(\widetilde{\Lambda})}\leq\frac{c}{\sqrt{1+x^{2}}}\|G^{(2)}_{-}-I\|_{H^{1}}\|G^{(2)}_{+}-I\|_{H^{1}}.\label{e7}
\end{align}
\end{lemma}
Equations \eqref{au1}--\eqref{au2} imply that analogous inequalities hold for $x\leq0$, though we omit their explicit form here. The Cauchy integral operators $\mathcal{C}_{\widetilde{\Lambda}\rightarrow\widetilde{\Lambda}'}^{\pm}$ are constructed by mapping $g$ to the boundary values $(\mathcal{C}_{\widetilde{\Lambda}}g)_{\pm}$ on $\widetilde{\Lambda}'$. In particular,
$\mathcal{C}_{\widetilde{\Lambda}}^{\pm}=\mathcal{C}_{\widetilde{\Lambda}\rightarrow\widetilde{\Lambda}}^{\pm}$.
\begin{proposition}\label{refa}
Suppose that $G^{(2)}_{\pm}-I\in H^{1}(\partial\widetilde{\Omega}^{\pm})\cap L^{2,1}(\partial\widetilde{\Omega}^{\pm})$
 and $z^{-1}(G^{(2)}_{\pm}-I )\in H^{1}(\partial\widetilde{\Omega}^{\pm})$.
Then the reconstructed potential satisfies $u\in H^{3}(\mathbb{R})\cap H^{2,1}(\mathbb{R})$.
\end{proposition}
\begin{proof}
Both \eqref{rrp} and \eqref{rrh} admit a decomposition
\begin{equation}\label{sgb}
\int_{\widetilde{\Lambda}}\omega^{(2)}\left(G^{(2)}_{x+}-G^{(2)}_{x-}\right)dz=F_{1}+F_{2}+F_{3}+F_{4},
\end{equation}
 with components defined as
\begin{equation*}
\begin{aligned}
&F_{1}=\int_{\widetilde{\Lambda}}\left(G^{(2)}_{x+}-G^{(2)}_{x-}\right)dz,\\
&F_{2}=\int_{\widetilde{\Lambda}}\left(\mathcal{C}_{G^{(2)}_{x}}I\right)\left(G^{(2)}_{x+}-G^{(2)}_{x-}\right)dz,\\
&F_{3}=\int_{\widetilde{\Lambda}}\left(\left(\mathcal{C}_{G^{(2)}_{x}}\right)^{2}I\right)\left(G^{(2)}_{x+}-G^{(2)}_{x-}\right)dz,\\
&F_{4}=\int_{\widetilde{\Lambda}}\left(\mathcal{C}_{G^{(2)}_{x}}\left(I-\mathcal{C}_{G^{(2)}_{x}}\right)^{-1}\left(\mathcal{C}_{G^{(2)}_{x}}\right)^{2}I\right)\left(G^{(2)}_{x+}-G^{(2)}_{x-}\right)dz.
\end{aligned}
\end{equation*}
By leveraging the properties of the Fourier and Laplace transforms, we conclude that $F_{1}\in L^{2,1}$. Since $F_{2}$ is a diagonal matrix, it does not contribute to the reconstruction of $p$ and $h$.

The estimation of $F_{3}$ and $F_{4}$ requires careful application of both the definition of the Cauchy integral operators $\mathcal{C}_{\widetilde{\Lambda}\rightarrow\widetilde{\Lambda}'}^{\pm}$ and their boundedness between the function spaces $L^{2}(\widetilde{\Lambda})$ and $L^{2}(\widetilde{\Lambda}')$. From \eqref{po1}--\eqref{po2}, we obtain
\begin{equation}
\begin{aligned}\nonumber
\left|F_{3}\right|\leq&\left|\left(\int_{\mathbb{R}}+\int_{\widetilde{\Lambda}_{2}}\right)\left(\mathcal{C}_{\widetilde{\Lambda}}^{+}\left(\mathcal{C}_{\widetilde{\Lambda}}^{-}\left(G^{(2)}_{x+}-I\right)\right)\left(I-G^{(2)}_{x-}\right)\right)\left(G^{(2)}_{x+}-I\right)dz\right|\\
&+\left|\left(\int_{\mathbb{R}}+\int_{\widetilde{\Lambda}_{4}}\right)\left(\mathcal{C}_{\widetilde{\Lambda}}^{-}\left(\mathcal{C}_{\widetilde{\Lambda}}^{+}\left(I-G^{(2)}_{x-}\right)\right)\left(G^{(2)}_{x+}-I\right)\right)\left(I-G^{(2)}_{x-}\right)dz\right|\\
\leq&\left|\int_{\mathbb{R}}\left(\mathcal{C}_{\widetilde{\Lambda}_{-}\rightarrow\mathbb{R}}^{+}\left(\mathcal{C}_{\widetilde{\Lambda}\rightarrow\widetilde{\Lambda}_{-}}^{-}\left(G^{(2)}_{x+}-I\right)\right)\left(I-G^{(2)}_{x-}\right)\right)\mathcal{C}_{\mathbb{R}}^{-}\left(G^{(2)}_{x+}-I\right)dz\right|\\
&+\left|\int_{\widetilde{\Lambda}_{2}}\left(\mathcal{C}_{\widetilde{\Lambda}_{-}\rightarrow\widetilde{\Lambda}_{2}}^{+}\left(\mathcal{C}_{\widetilde{\Lambda}\rightarrow\widetilde{\Lambda}_{-}}^{-}\left(G^{(2)}_{x+}-I\right)\right)\left(I-G^{(2)}_{x-}\right)\right)\left(G^{(2)}_{x+}-I\right)dz\right|\\
&+\left|\int_{\mathbb{R}}\left(\mathcal{C}_{\widetilde{\Lambda}_{+}\rightarrow\mathbb{R}}^{-}\left(\mathcal{C}_{\widetilde{\Lambda}\rightarrow\widetilde{\Lambda}_{+}}^{+}\left(I-G^{(2)}_{x-}\right)\right)\left(G^{(2)}_{x+}-I\right)\right)\mathcal{C}_{\mathbb{R}}^{+}\left(I-G^{(2)}_{x-}\right)dz\right|\\
&+\left|\int_{\widetilde{\Lambda}_{4}}\left(\mathcal{C}_{\widetilde{\Lambda}_{+}\rightarrow\widetilde{\Lambda}_{4}}^{-}\left(\mathcal{C}_{\widetilde{\Lambda}\rightarrow\widetilde{\Lambda}_{+}}^{+}\left(I-G^{(2)}_{x-}\right)\right)\left(G^{(2)}_{x+}-I\right)\right)\left(I-G^{(2)}_{x-}\right)dz\right|.
\end{aligned}
\end{equation}
By applying estimates \eqref{e1}--\eqref{e7} established in Lemma \ref{imlem}, we can further derive the following bound for $\left|F_{3}\right|$
\begin{align*}
\left|F_{3}\right|\lesssim&\left\|\left(\mathcal{C}_{\widetilde{\Lambda}\rightarrow\widetilde{\Lambda}_{-}}^{-}\left(G^{(2)}_{x+}-I\right)\right)\right\|_{L^{2}(\widetilde{\Lambda}_{-})}\left(\left\|\mathcal{C}_{\mathbb{R}}^{-}\left(G^{(2)}_{x+}-I\right)\right\|_{L^{2}(\mathbb{R})}+\left\|G^{(2)}_{x+}-I\right\|_{L^{2}(\widetilde{\Lambda}_{2})}\right)\\
&+\left\|\left(\mathcal{C}_{\widetilde{\Lambda}\rightarrow\widetilde{\Lambda}_{+}}^{+}\left(I-G^{(2)}_{x-}\right)\right)\right\|_{L^{2}(\widetilde{\Lambda}_{+})}\left(\left\|\mathcal{C}_{\mathbb{R}}^{+}\left(I-G^{(2)}_{x-}\right)\right\|_{L^{2}(\mathbb{R})}+\left\|I-G^{(2)}_{x-}\right\|_{L^{2}(\widetilde{\Lambda}_{4})}\right)\\
\lesssim&\,\frac{1}{1+x^{2}}.
\end{align*}
The fourth integral component $F_{4}$ admits the representation
\begin{equation*}
F_{4}=\int_{\widetilde{\Lambda}}\left(\mathcal{C}_{G^{(2)}_{x}}f\right)\left(G^{(2)}_{x+}-G^{(2)}_{x-}\right)dz,
\end{equation*}
with $f=(I-\mathcal{C}_{G^{(2)}_{x}})^{-1}(\mathcal{C}_{G^{(2)}_{x}})^{2}I$. The boundedness property of $(I-\mathcal{C}_{G^{(2)}_{x}})^{-1}$ combined with \eqref{e7} establishes the estimate
\begin{equation*}
\|f\|_{L^{2}(\widetilde{\Lambda})}\lesssim\|(\mathcal{C}_{G^{(2)}_{x}})^{2}I\|_{L^{2}(\widetilde{\Lambda})}\lesssim\frac{1}{\sqrt{1+x^{2}}}.
\end{equation*}
It follows from \eqref{po1}--\eqref{po2} that
\begin{equation*}
\begin{aligned}
\left|F_{4}\right|\leq&\left|\left(\int_{\mathbb{R}}+\int_{\widetilde{\Lambda}_{2}}\right)\left(\mathcal{C}_{\widetilde{\Lambda}}^{+}f\left(I-G^{(2)}_{x-}\right)\right)\left(G^{(2)}_{x+}-I\right)dz\right|\\
&+\left|\left(\int_{\mathbb{R}}+\int_{\widetilde{\Lambda}_{4}}\right)\left(\mathcal{C}_{\widetilde{\Lambda}}^{-}f\left(G^{(2)}_{x+}-I\right)\right)\left(I-G^{(2)}_{x-}\right)dz\right|\\
\leq&\left|\int_{\mathbb{R}}\left(\mathcal{C}_{\widetilde{\Lambda}\rightarrow\mathbb{R}}^{+}f\left(I-G^{(2)}_{x-}\right)\right)\mathcal{C}_{\mathbb{R}}^{-}\left(G^{(2)}_{x+}-I\right)dz\right|\\
&+\left|\int_{\widetilde{\Lambda}_{2}}\left(\mathcal{C}_{\widetilde{\Lambda}\rightarrow\widetilde{\Lambda}_{2}}^{+}f\left(I-G^{(2)}_{x-}\right)\right)\left(G^{(2)}_{x+}-I\right)dz\right|\\
&+\left|\int_{\mathbb{R}}\left(\mathcal{C}_{\widetilde{\Lambda}\rightarrow\mathbb{R}}^{-}f\left(G^{(2)}_{x+}-I\right)\right)\mathcal{C}_{\mathbb{R}}^{+}\left(I-G^{(2)}_{x-}\right)dz\right|\\
&+\left|\int_{\widetilde{\Lambda}_{4}}\left(\mathcal{C}_{\widetilde{\Lambda}\rightarrow\widetilde{\Lambda}_{4}}^{-}f\left(G^{(2)}_{x+}-I\right)\right)\left(I-G^{(2)}_{x-}\right)dz\right|.
\end{aligned}
\end{equation*}
Then, an application of Lemma \ref{imlem} yields
\begin{equation*}
\begin{aligned}
\left|F_{4}\right|
\lesssim&\left\|f\left(I-G^{(2)}_{x-}\right)\right\|_{L^{2}(\widetilde{\Lambda})}\left(\left\|\mathcal{C}_{\mathbb{R}}^{-}\left(G^{(2)}_{x+}-I\right)\right\|_{L^{2}(\mathbb{R})}+\left\|G^{(2)}_{x+}-I\right\|_{L^{2}(\widetilde{\Lambda}_{2})}\right)\\
&+\left\|f\left(G^{(2)}_{x+}-I\right)\right\|_{L^{2}(\widetilde{\Lambda})}\left(\left\|\mathcal{C}_{\mathbb{R}}^{+}\left(I-G^{(2)}_{x-}\right)\right\|_{L^{2}(\mathbb{R})}+\left\|I-G^{(2)}_{x-}\right\|_{L^{2}(\widetilde{\Lambda}_{4})}\right)\\
\lesssim&\,\frac{1}{1+x^{2}}.
\end{aligned}
\end{equation*}
Base on the above analysis for $F_{i}$, $1\leq i\leq4$, we conclude that $p,h\in L^{2,1}(\mathbb{R})$. Proposition \ref{timep} implies that $z^{-1}(G^{(2)}_{\pm}-I )\in H^{1}(\partial\widetilde{\Omega}^{\pm})$. Thus, by repeating the same analysis to \eqref{rru}, we can obtain $u\in L^{2,1}(\mathbb{R})$.

Finally, we establish that $h_{x}\in L^{2}(\mathbb{R})$. Defining the transform $L(x,z)=e^{i\int_{x}^{\infty}|u_{y}(y)|^{2}dy\sigma_{3}}\Psi(x,z)$, it follows from \eqref{Ps} that
\begin{equation}\label{ahf}
L_{x}=\frac{i}{2}z\left[\sigma_{3},L\right]+\widetilde{Q}L,
\end{equation}
where
\begin{equation*}
\widetilde{Q}=e^{i\int_{x}^{\infty}|u_{y}(y)|^{2}dy\widehat{\sigma}_{3}}\begin{pmatrix}0&iu_{x}\\
-\bar{u}_{xx}-i\bar{u}_{x}|u_{x}|^{2}&0\end{pmatrix}.
\end{equation*}
Moreover, $\omega^{(2)}$ satisfies \eqref{ahf}, leading to the derivative relation
\begin{equation*}
\frac{d}{dx}\left(\omega^{(2)} \left(G^{(2)}_{x+}-G^{(2)}_{x-}\right)\right)=\frac{i}{2}z\left[\sigma_{3},\omega^{(2)} \left(G^{(2)}_{x+}-G^{(2)}_{x-}\right)\right]+\widetilde{Q}\,\omega^{(2)} \left(G^{(2)}_{x+}-G^{(2)}_{x-}\right).
\end{equation*}
By integrating both sides of this equation, we derive
\begin{equation}
\frac{d}{dx}\int_{\widetilde{\Lambda}}\omega^{(2)} \left(G^{(2)}_{x+}-G^{(2)}_{x-}\right)dz=I_{1}(x)+I_{2}(x),
\end{equation}
where
\begin{equation*}
\begin{aligned}
&I_{1}(x)=\frac{i}{2}\int_{\widetilde{\Lambda}}\left[\sigma_{3},\omega^{(2)} z\left(G^{(2)}_{x+}-G^{(2)}_{x-}\right)\right]dz,\\
&I_{2}(x)=\int_{\widetilde{\Lambda}}\widetilde{Q}\,\omega^{(2)} \left(G^{(2)}_{x+}-G^{(2)}_{x-}\right)dz.
\end{aligned}
\end{equation*}
The term $I_{1}$ admits a decomposition analogous to \eqref{sgb}, yielding $I_{1}\in L^{2}(\mathbb{R})$  under the condition that $z(G^{(2)}_{\pm}-I)\in L^{2}(\partial \widetilde{\Omega}^{\pm})$. Analysis of \eqref{sgb} establishes $\widetilde{Q}\in L^{\infty}(\mathbb{R})$, which consequently implies $I_{2}\in L^{2}(\mathbb{R})$. These results collectively demonstrate that $h_{x}\in L^{2}(\mathbb{R})$.
\end{proof}

\subsection{Proof of Theorem \ref{thm}}\label{sub43}

We are now ready to prove  Theorem \ref{thm} as follows.

\begin{proof}
Given initial data $u_{0}\in H^{3}(\mathbb{R})\cap H^{2,1}(\mathbb{R})$, Proposition \ref{timep} implies that $G^{(2)}_{\pm}(t,\cdot)-I \in H^{1}(\partial\widetilde{\Omega}^{\pm})\cap L^{2,1}(\partial\widetilde{\Omega}^{\pm})$ and $z^{-1}(G^{(2)}_{\pm}(t,\cdot)-I ) \in H^{1}(\partial\widetilde{\Omega}^{\pm})$ hold for all $t\in [-T,T]$ with $T>0$. Moreover, Proposition \ref{con} and the solution verification process in \cite{Chengglo} guarantee that the reconstructed potential satisfies \eqref{fl}. Furthermore, the inverse scattering estimate in Proposition \ref{refa} demonstrates that this potential $u$ belongs to $H^{3}(\mathbb{R})\cap H^{2,1}(\mathbb{R})$. Finally, the Lipschitz continuity of the solution map follows from Proposition \ref{timep}, \eqref{lpxz} and the proofs of Proposition \ref{refa}.
\end{proof}

\section*{Acknowledgements}
\addcontentsline{toc}{section}{Acknowledgements}
This work is supported by the National Natural Science Foundation of China   {(Grant No. 12271104, 1240011471)}.

\end{document}